\documentclass{amsart}

\usepackage{bbm}
\usepackage{amssymb}
\usepackage{amsmath}
\usepackage{xfrac}
\usepackage{array}
\usepackage{mathtools}
\usepackage{stmaryrd}
\usepackage{enumitem}
\usepackage{hyperref}
\usepackage{tikz-cd}
\usepackage{amsthm}
\usepackage{color}

\newtheorem{prop}{Proposition}
\newtheorem{definition}[prop]{Definition}
\newtheorem{theorem}[prop]{Theorem}
\newtheorem*{theorem*}{Theorem}

\newtheorem{thm}[prop]{Theorem}
\newtheorem{cor}[prop]{Corollary}
\newtheorem{lemma}[prop]{Lemma}
\newtheorem{remark}[prop]{Remark}

\numberwithin{prop}{section}

\newcommand{\RR}{\ensuremath{\mathbb{R}}}
\newcommand{\CC}{\ensuremath{\mathbb{C}}}

\newcommand{\NN}{\ensuremath{\mathbb{N}}}
\newcommand{\HH}{\ensuremath{\mathbb{H}}}
\newcommand{\ZZ}{\ensuremath{\mathbb{Z}}}

\newcommand{\SL}[2]{\textrm{SL}_{#1}\left(#2\right)}

\newcommand{\ti}[1]{\ensuremath{\tilde{#1}}}

\newcommand{\CAT}{\ensuremath{\mathrm{CAT}(0)}}

\newcommand{\Rips}{\ensuremath{\mathrm{R}}}

\DeclarePairedDelimiter\abs{\lvert}{\rvert}

\newcommand{\nrm}{\ensuremath{\vartriangleleft }}

\DeclarePairedDelimiter\gener{\langle}{\rangle}
\newcommand{\Is}[1]{\ensuremath{\mathrm{Isom}(#1)}}
\newcommand{\centralizer}[2]{\mathrm{C}_{#1}\left(#2\right)}
\newcommand{\Sub}[1]{\mathrm{Sub}\left(#1\right)}
\newcommand{\Subb}[2]{\mathrm{Sub}_{#2}(#1)}
\newcommand{\IRS}[1]{\mathrm{IRS}\left(#1\right)}

\title{Local Rigidity Of Uniform Lattices}

\author{Tsachik Gelander and Arie Levit}

\begin{document}

\begin{abstract}
	
	We establish  topological local rigidity for uniform lattices in compactly generated groups, extending the result of Weil from the realm of Lie groups. 
	We generalize the classical local rigidity theorem of Selberg, Calabi and Weil to irreducible uniform lattices in $\text{Isom}(X)$ where $X$ is a proper $\text{CAT}(0)$ space with no Euclidian factors, not isometric to the hyperbolic plane. 
	We deduce an analog of Wang's finiteness theorem for certain non-positively curved metric spaces.

	%

	%
	%

	%
	%
\end{abstract}

\maketitle

\section{Introduction}

Over the last fifty years, 
the study of rigidity phenomena  played a central role in group theory and its interplay with number theory, geometry and dynamics.
Three fundamental results of the theory are local rigidity, strong rigidity and super-rigidity. 

In the last two decades there has been a comprehensive, and quite successful, effort to extend this theory from the classical setting of Lie groups and  symmetric spaces to the much wider framework of locally compact groups and non-positively curved metric spaces. Much of the attempt has been dedicated towards generalising the  latter two phenomena, namely Mostow strong rigidity and Margulis super-rigidity, see e.g. \cite{burger2000groups,burger2000lattices,shalom2000rigidity,Mo,GKM,CM2,bader2013algebraic,bader2013integrable,bader2014boundaries,shalom2013commensurated}. 
The present paper focuses on {\it local rigidity}.

\subsection*{Overview}

Let $\Gamma$ be a uniform lattice in the locally compact group $G$. 
The lattice $\Gamma$ is  \emph{topologically locally rigid} if every homomorphism from $\Gamma$ to $G$ which is sufficiently close  to the inclusion mapping  in the point-wise convergence topology is  injective and its image is a uniform lattice  in $G$.
The lattice  $\Gamma$ is 
\emph{locally rigid} if every such homomorphism is in fact conjugate in $G$ to the inclusion mapping.
See \S \ref{sub:the representation space} for the precise definitions.

One main goal of this paper is to establish  topological local rigidity for uniform lattices in general compactly generated groups  --- extending Weil's theorem \cite{We60} from the realm of Lie groups to  locally compact groups. Another main goal is to generalize the classical local rigidity theorems of Selberg, Calabi and Weil \cite{Selberg, Calabi, Weil2} to the general context of $\CAT$ groups.  Furthermore we study local rigidity from another perspective, in terms of the Chabauty topology.

We  apply these results to obtain new finiteness statements on non-positively curved metric spaces in analogy with Wang's classical finiteness theorem.

\subsection*{Topological Local   Rigidity} A. Weil proved that uniform lattices in connected Lie groups are  topologically locally rigid \cite{We60}. 
We complete the picture by extending  this result to all compactly generated groups.

\begin{theorem}
	\label{thm:local rigidity for lc groups}
	Uniform lattices in compactly generated groups are  topologically  locally rigid.
\end{theorem}

It is rather straightforward that a small deformation of a uniform subgroup remains uniform. The main point is in showing that it stays  discrete and injective.

Compact generation of $G$ is equivalent to  finite generation of a uniform lattice $\Gamma$ in $G$. Without this requirement lattices need not be  topologically locally rigid. For instance, let $G=\Gamma$ be an infinite rank free group and note that the identity map $i:\Gamma\to \Gamma$ is approximated in the representation space by non-injective maps sending all but finitely many generators to the trivial element.\footnote{Note that there are examples of  topologically  locally rigid lattices in non-compactly generated groups. For instance, the additive group $\ZZ\left[\frac{1}{2}\right]$ with the discrete topology is  topologically locally rigid regarded as a lattice in itself.} 

Let us recall that a topological group admitting a uniform lattice is automatically locally compact \cite[2.2]{montgomery}.


Weil's proof relied on the  differentiable structure of $G$ as well as the connectedness and the existence of a simply connected cover in the classical case of Lie groups.  These notions are not available in the general context under consideration.
Instead, our approach is to make use of the action of $G$ on an appropriate Rips complex. 

The Rips complex is simply connected whenever $G$ is compactly presented and for this reason we first prove Theorem \ref{thm:local rigidity for lc groups} in \S \ref{sec:compactly presented} under this additional assumption and then explain in \S \ref{sec:locally compact case} how it  can be removed.  \S \ref{sec:preliminaries} is dedicated to several geometric preliminaries that are required for the proof. \S \ref{sec:totally disconnected groups} contains an independent proof of  topological local rigidity for certain compactly presented  totally disconnected groups, having the advantage of being rather elementary while providing an insight into the more complicated general case.



\subsection*{Local Rigidity}

Local rigidity was first proved by Selberg \cite{Selberg} for uniform
lattices in  $\SL{n}{\RR }$, $n\geq 3$  and by Calabi \cite{Calabi}
for uniform lattices in  $\mathrm{PO}(n,1)=\Is{\HH^n}$, $n\geq 3$.
Weil \cite{Weil2}
generalized these results to uniform irreducible lattices in any connected semisimple Lie group $G$,
assuming 
that $G$ is not locally isomorphic to $\SL{2}{\RR}$, the isometry group of the hyperbolic plane $\HH^2$. In the latter case lattices have many non-trivial deformations.

Quite surprisingly, it turns out that $\mathbb{H}^2$ is the only obstruction to local rigidity of uniform lattices in the much greater generality of irreducible $\CAT$ spaces.


\begin{theorem}
	\label{thm:local rigidity of uniform in general cat0-intro}
	Let $X$ be a proper geodesically complete $\CAT$ space without Euclidean
	factors and with $\Is{X}$ acting cocompactly. 
	Let $\Gamma$ be a uniform lattice
	in $\Is{X}$. Assume that for every de Rham factor $Y$ of $X$ isometric
	to the hyperbolic plane the projection of $\Gamma$ to $\Is{Y}$
	is non-discrete. Then $\Gamma$ is locally rigid.
\end{theorem}

As in the works of Selberg, Calabi and Weil  topological local rigidity plays an important role in the proof of local rigidity.  Theorem \ref{thm:local rigidity for lc groups} allows us to apply the Caprace--Monod theory of $\CAT$ isometry groups and their lattices \cite{CM1,CM2} reducing the question to an irreducible lattice in a product of a semisimple Lie group and a totally disconnected group. In the purely Lie group case we   rely on the above mentioned classical results, while the purely totally disconnected case is treated in a somewhat more general context in \S\ref{sec:totally disconnected groups}. Finally,  the case where both factors are non-trivial  relies on the superrigidity theorem  of Monod \cite{Mo}.



\subsection*{The Chabauty space and local rigidity}

Consider the space of closed subgroups $\Sub{G}$ of a topological group $G$ equipped with the Chabauty topology.
This  suggests a different approach to local rigidity --- we say that a uniform lattice is \emph{Chabauty locally rigid} if it admits a Chabauty neighborhood consisting of isomorphic uniform lattices. For Lie groups this notion was first considered by Macbeath \cite{macbeath}. We prove Chabauty local rigidity for a rather broad family of  groups.

\begin{theorem}
\label{thm:Chabauty local rigidity for compacty presented groups}
Let $G$ be a compactly generated group without non-trivial compact normal subgroups.  Then the collection of uniform lattices is Chabauty open.   If moreover $G$ is compactly presented then uniform lattices are Chabauty locally rigid.
\end{theorem}

See Theorem \ref{thm:Chaubuty local rigidity} for a sharper statement of this result. 
In view of Proposition \ref{prop:C is continuous at a uniform lattice} the notion of Chabauty local rigidity is stronger than topological local rigidity. The compact generation assumption is necessary --- see Proposition \ref{prop:condition for a uniform lattice to admit a Chabauty neighborhood of cocompact subgroups}.
Isometry groups of certain $\CAT$ spaces are a special case of Theorem \ref{thm:Chabauty local rigidity for compacty presented groups}.

\begin{cor}
	\label{cor:local rigidity in the Chabauty sense-intro}
	Let $X$ be a proper  geodesically complete $\CAT$ space with $\Is{X}$ acting cocompactly. Then uniform lattices in $\Is{X}$ are Chabauty locally rigid.
\end{cor}

Whenever the $\CAT$ space $X$ has no Euclidean factors and the lattice $\Gamma$ projects non-discretely to the isometry group of every hyperbolic plane factor of $X$, Theorem \ref{thm:local rigidity of uniform in general cat0-intro} implies that the Chabauty neighborhood of $\Gamma$ given in Corollary \ref{cor:local rigidity in the Chabauty sense-intro} consists  of conjugates of $\Gamma$. This is discussed in Corollary \ref{cor:occ for CAT0 groups} below. In that situation $\Gamma$ is  \emph{OCC} or \emph{open conjugacy class subgroup} in the sense of \cite{glasner}. 





\subsection*{Finiteness of lattices of bounded covolume}

A well-known classical result that relies on local rigidity is the finiteness theorem of Wang. It says that a connected semisimple Lie group $G$ not locally isomorphic to $\textrm{SL}(2,\RR)$ or $\textrm{SL}(2,\CC)$ has only finitely many conjugacy classes of irreducible lattices  of co-volume $\le v $ for every $v>0$. We refer to  Wang's original proof \cite{wangtopics} and to \cite[Section 13]{HV} for a treatment of the case where $G$ has factors isomorphic to $\textrm{SL}(2,\RR)$ or $\textrm{SL}(2,\CC)$. 

The two main ingredients of Wang's proof are: 

\begin{enumerate}
	\item Local rigidity of all lattices in $G$.
	\item The Kazhdan--Margulis theorem \cite{KM}\footnote{See Definition \ref{def:KM property} and the paragraph following it for a discussion of this theorem. }.
\end{enumerate}


Indeed, $(2)$ combined with the Mahler--Chabauty compactness criterion implies that the set of conjugacy classes of lattices of co-volume $\le v$ is compact with respect to the Chabauty topology, while $(1)$ implies that it is discrete. However, even if one restricts attention to uniform lattices only, it is crucial that $(1)$ holds for all lattices\footnote{Recall that when $G$ is a semisimple Lie group not locally isomorphic to $\textrm{SL}(2,\RR)$ or $\textrm{SL}(2,\CC)$ then also the non-uniform irreducible lattices are locally rigid.  For rank one groups this was proved by Garland and Raghunathan \cite{Gar-Rag}. (An alternative geometric proof for the rank one case is given in \cite{Bergeron-Gelander}.) For higher rank groups this is a consequence of Margulis' super-rigidity theorem \cite{Margulis}.}  in $G$. 
This is illustrated in the case of $\textrm{SL}(2,\CC)$ where uniform lattices are locally rigid while the finiteness statement fails. The issue is that a sequence of  uniform lattices of bounded co-volume may converge to a non-uniform lattice which is not locally rigid. 

We obtain a similar finiteness result for $\CAT$ groups. However, since local
rigidity in the general $\CAT$ context is established for uniform lattices
only, we need to make an additional assumption.

\begin{definition}
	\label{def:uniformly discrete}
	Let $G$ be a locally compact group. A family $\mathcal{F}$ of lattices in
	$G$ is \emph{uniformly discrete} if there is an identity neighborhood $U
	\subset G$ such that $\Gamma^g \cap U = \{e\}$ holds for every lattice $\Gamma
	\in \mathcal{F}$ and every $g\in G$. 
\end{definition}

Lattices belonging to a uniformly discrete family\footnote{Uniform discreteness is not to be confused with a weaker notion of a \emph{joint discreteness}  given in Definition \ref{def:jointly discrete}. See also Footnote \ref{footnote:on jointly discrete} concerning this terminology. } are all uniform \cite[1.12]{Rag}, and their co-volumes are clearly bounded from below.

\begin{theorem}
	\label{thm:wang-finiteness-intro}
	Let $X$ be a proper geodesically complete $\CAT$ space without Euclidean
	factors and with $\Is{X}$ acting cocompactly. 
	Let $\mathcal{F}$  be a uniformly
	discrete family of lattices in $\Is{X}$ so that every $\Gamma \in \mathcal{F}$
	projects non-discretely to the isometry group of every hyperbolic plane factor.
	Then $\mathcal{F}$ admits only finitely many lattices up to conjugacy
	with $\text{co-vol} \le v$ for every fixed $v > 0$.
\end{theorem}

The proof follows rather immediately from the open conjugacy class property, Corollary \ref{cor:occ for CAT0 groups}, 
and is perhaps more conceptual than Wang's original argument.

\subsection*{Finiteness of lattices containing a given lattice}
Prior to the finiteness theorem mentioned above, Wang proved a weaker result that applies in a more general situation. Namely, he showed that a lattice in a semisimple Lie group is contained in only finitely many lattices \cite{wang_finitely}.
This result  relies on the Kazhdan--Margulis theorem as well. To obtain a suitable generalization to locally compact groups we introduce the following notion.

\begin{definition}
	\label{def:KM property}
	Let $G$ be a locally compact group. A family $\mathcal{F}$ of lattices in $G$ has property $(KM)$ if there is an identity neighborhood
	$U \subset G$ such that for every lattice $\Gamma\in\mathcal{F}$ there is an element $g\in G$ satisfying $g\Gamma g^{-1}\cap U=\{ e \}$. We will say that  $G$ has property $(KM)$ if the family of all lattices in $G$ has property $(KM)$.
\end{definition}

Kazhdan and Margulis showed \cite{KM} that a semisimple Lie group
with no compact factors has property $(KM)$. We refer to \cite{WUD}
for a short proof of the Kazhdan--Margulis theorem and for additional examples
of groups with property $(KM)$.

\begin{thm}\label{thm:W1KM-intro}
	Let $G$ be a compactly generated locally compact group. Assume that every lattice in $G$ has a trivial centralizer. 
	Let $\mathcal{F}$ be a family of lattices in $G$ with property  $(KM)$ and admitting some $\Gamma \in \mathcal{F}$ that is finitely generated and is a least element in $\mathcal{F}$ with respect to inclusion. Then $\mathcal{F}$ is finite.
\end{thm}

Examples of groups in which all lattices have trivial centralizers include second countable locally compact groups with a trivial amenable radical \cite[5.3]{caprace2017indicability} as well as the isometry groups of certain $\CAT$ spaces. See \S \ref{sub:cat0 lattices and their properties} for more details.   

%

In particular, it follows that a family of lattices $\mathcal{F}$ containing a given  finitely generated lattice $\Gamma$ and satisfying $\Gamma' \cap U = \{e\}$ for some identity neighborhood $U$ and all $\Gamma' \in \mathcal{F}$ must be finite.

We would like to mention the work  \cite{bass1990uniform} that contains somewhat related finiteness results for uniform tree lattices. 
  Recently Burger and Mozes  have established the analogue of \cite{wang_finitely} for certain uniform lattices in products of trees by  carefully studying discrete groups  containing a given irreducible uniform lattice.

\begin{remark}
	\S \ref{sec:lattices containing a given lattice} dealing with these generalizations of \cite{wang_finitely} does not depend on local rigidity and is the only part of the current paper that extends without difficulty  to non-uniform lattices. 
\end{remark}


\subsection*{Invariant random subgroups}

Finally, we investigate several questions concerning Chabauty spaces in general. For instance,  the \emph{invariant random subgroup} corresponding to a   uniform lattice $\Gamma$ is shown to depend continuously on $\Gamma$ in the Chabauty topology; see Proposition \ref{prop:Ulat space and irs}. This relies on the Chabauty-continuity of the co-volume of uniform lattices   established in Proposition \ref{prop:covolume is continous}.




\subsection*{Acknowledgments}

Pierre-Emmanuel Caprace has  kindly and generously provided us with  clever ideas and arguments  extending the generality and improving our results. His suggestions also helped to improve the presentation of this paper. We would like to thank him for this immensely valuable service.

 In particular Theorem \ref{thm:Chabauty local rigidity for compacty presented groups}, Proposition \ref{prop:condition for a uniform lattice to admit a Chabauty neighborhood of cocompact subgroups}, Corollaries \ref{cor:local topological rigidity for t.d.c.p. groups} and \ref{cor:local topological rigidity for product t.d.c.p. groups}, Remarks \ref{remark:uniform tree lattices} and \ref{remark:another proof of local rigidity}, Theorem \ref{thm:Chaubuty local rigidity}, Proposition \ref{prop:uniform lattice has neighborhood of uniform lattices}, Remark \ref{remark: on normalizers in CAT0} and Proposition \ref{prop:Ulat space and irs} all contain arguments suggested by Caprace.

We extend our gratitude to Nir Lazarovich for  helpful conversations in which some arguments of Proposition \ref{prop:acting cocompactly is open for finitely generated} were established.




\setcounter{tocdepth}{1}
\tableofcontents

\section{Deformations of lattices}

\subsection{The representation space}
\label{sub:the representation space}


We describe a topological space classifying the representations of a given discrete group into a target topological group. The notion of local rigidity for a lattice is then defined in terms of that space.

\begin{definition}
\label{def:representaion space}
Let $\Gamma$ be a discrete group and $G$ a topological group.  The \textbf{representation space} $\mathcal{R}(\Gamma,G)$ is the space of all group homomorphisms from $\Gamma$ to $G$ with the point-wise convergence topology.
\end{definition}

Note that for any point $r\in\mathcal{R}(\Gamma,G)$ the group $r(\Gamma)$ is a quotient of $\Gamma$. 

If $\Gamma$ has a finite generating set $\Sigma$ then $\mathcal{R}(\Gamma, G)$ is homeomorphic to a closed subspace of $G^\Sigma$ with the product topology. This subspace is determined by closed conditions arising from the relations among the generators in $\Sigma$.

Given  a discrete subgroup $\Gamma$ of $G$ the inclusion morphism $ \Gamma \hookrightarrow G$ can naturally be regarded as a point of $\mathcal{R}(\Gamma,G)$. We will denote this point by $r_0$ and frequently refer to neighborhoods of $r_0$ in $\mathcal{R}(\Gamma,G)$ as \emph{deformations} of $\Gamma$ in $G$.

Recall that a discrete subgroup $\Gamma \le G$ is a \emph{uniform lattice} in $G$ if there is a compact subset $K \subset G$ such that $G = \Gamma K$. Note that $\Gamma$ is finitely generated if and only if $G$ is compactly generated \cite[5.C.3]{cor_har}.

\begin{definition}
\label{def:local topological rigidity}
The lattice $\Gamma$ is \textbf{ topologically  locally rigid} if the inclusion morphism $r_0$ admits an open neighborhood $\mathcal{U}$ in $\mathcal{R}(\Gamma,G)$ such  that for every $r \in \mathcal{U}$ the subgroup $r(\Gamma) \le G$ is a uniform lattice and $r$ is injective.
\end{definition}

Less formally, a uniform lattice is  topologically  locally rigid if sufficiently small deformations are injective, discrete and co-compact. 


\begin{definition}
\label{def:local rigidity}
The lattice $\Gamma$ is \textbf{locally rigid} if  sufficiently small deformations of it are given by conjugation in $G$.
\end{definition}

\subsubsection*{Local rigidity inside a larger group}
There is a more general notion of local rigidity requiring small deformations of $\Gamma$  to arise from (not necessarily inner) automorphisms of $G$. The following is one particular example of this.

\begin{definition}
\label{def:be locally rigid in}
Assume that $G$ is a closed subgroup of $G^\dagger$. The lattice $\Gamma$ is \textbf{locally rigid in  $G^\dagger$} if sufficiently small deformations inside $\mathcal{R}(\Gamma,G)$  arise from conjugation by an element of $G^\dagger$.
\end{definition}

If $\Gamma$ is locally rigid in some larger group $G^\dagger$ then it is in particular  topologically locally rigid. Moreover, if $G$ is cocompact in $G^\dagger$ and all uniform lattices in $G^\dagger$ are known to be locally rigid then  uniform lattices in $G$ are locally rigid in $G^\dagger$ in the sense of Definition \ref{def:be locally rigid in}.

\begin{prop}
\label{pro:locally rigid in and dense projections}
Let $G_i$ be a closed and cocompact subgroup of the locally compact group $G_i^\dagger$ for $i=1,2$. Let $\Gamma$ be a uniform lattice in   $  G_1 \times G_2$   projecting densely to both factors. If $\Gamma$ is locally rigid in $ G^\dagger_1 \times G^\dagger_2$ then $\Gamma$ is locally rigid in $N_{G^\dagger_1 \times G^\dagger_2}(G_1 \times G_2)$.
\end{prop}
\begin{proof}
To simplify notation let
$$ G = G_1 \times G_2, \quad G^\dagger = G^\dagger_1 \times G^\dagger_2 .$$
Consider a deformation $r \in \mathcal{R}(\Gamma,G)$ which is sufficiently close to the inclusion    so that $r(\gamma) = \gamma^g$ for some $g = (g_1, g_2) \in G_1^\dagger \times G_2^\dagger = G^\dagger$ and all $\gamma \in \Gamma$. We have that
$$ \overline{\mathrm{proj}_i(r(\Gamma))} = \overline{\mathrm{proj}_i(g^{-1} \Gamma g)} = g_i^{-1} \overline{\mathrm{proj}_i(\Gamma)} g_i = g_i^{-1} G_i g_i $$
On the other hand $r(\Gamma) \subset G$ and therefore $\overline{\mathrm{proj}_i(r(\Gamma))} \le G_i$. This means  that $g  \in G ^\dagger$ conjugates $G $ to a subgroup of itself. 

Consider the Chabauty space $\Sub{G^\dagger}$ with the natural containment partial order. This partial order is semicontinuous on $\Sub{G^\dagger}$ in the sense that given $H_1 \le H_2$  there are Chabauty open neighborhoods $H_1 \in \Omega_1$ and $H_2 \in \Omega_2$ so that
$$ H'_1 \le H_2 \quad \forall H'_1 \in \Omega_1 \quad \text{and} \quad H_1 \le H'_2 \quad \forall H'_2\in\Omega_2 $$
The $G^\dagger$-conjugacy class of $G$ is a compact subset of the Chabauty space $\Sub{G ^\dagger}$ and therefore it admits minimal and maximal elements with respect to containment \cite{ward1954partially}. In particular, every group in the $G^\dagger$-conjugacy class of $G$ is both minimal and maximal  so that  $g^{-1} G g = G$ and  $g \in N_{G^\dagger}(G)$ as required.
\end{proof}



\subsubsection*{Local rigidity and finite index subgroups}

In certain situations, it is convenient to replace the group $\Gamma$ by a finite index subgroup having additional properties. The following lemma allows this procedure.


\begin{lemma}
\label{lem:local rigidity and finite index}
Let $G$ be a topological group, $\Gamma \le G$ a finitely generated lattice and $\Delta \nrm \Gamma$  a normal subgroup of finite index. Assume that $\centralizer{G}{\Delta}$ is trivial. \(  \)If $\Delta$ is locally rigid then $\Gamma$ is locally rigid as well.
\end{lemma}
\begin{proof}
Let $r \in \mathcal{R}(\Gamma,G)$ be a sufficiently small deformation of $\Gamma$ so that  $r(\delta) = g^{-1}\delta g$ for some $g \in G$ and all $\delta \in \Delta$. 
For every pair of elements 
$\gamma \in \Gamma$ and $\delta \in \Delta$ we have
$$ r(\delta)^{r(\gamma)} = r(\delta ^ \gamma) = \delta^{\gamma g} = r(\delta)^{g^{-1} \gamma g}$$
The group $\Delta$ is finitely generated, being a finite index subgroup of $\Gamma$. Applying the above argument with $\delta$ ranging over a finite generating set for $\Delta$ gives
$$ g^{-1} \gamma g r(\gamma)^{-1} \in \centralizer{G}{r(\Delta)} = \centralizer{G}{\Delta}^g = \{e\}$$
which means that $\gamma^g = r(\gamma)$ for all $\gamma \in \Gamma$ as required.
\end{proof}


\subsection{The Chabauty topology}
\label{sub:The Chabauty topology}

There is another related viewpoint on local rigidity, expressed in terms of the Chabauty space of closed subgroups.

\begin{definition}
        \label{def:the Chabauty topology}
        Given a locally compact group $G$ let $\Sub{G}$ denote the space  of closed subgroups of $G$ equipped with the \textbf{Chabauty topology} generated by the sub-basis consisting of
        \begin{itemize}
                \item $\mathcal{O}_1(U) = \{ \text{$H \le G$ closed } \: : \: H \cap U \neq \emptyset \} $ for every open subset $U \subset G$,
                \item $\mathcal{O}_2(K) = \{ \text{$H \le G$ closed } \: : \: H \cap K = \emptyset \} $ for every compact subset $K \subset G$.
        \end{itemize}
\end{definition}

The space $\Sub{G}$ is Chabauty compact. A reference to the Chabauty topology is \cite[VIII.$\S$5]{bourbaki2004integration}. One major difference between the representation space and the Chabauty topology is that the latter only takes into account the embedding of a  given subgroup as a closed subspace of $G$, a priori irrespective of the algebraic structure of that subgroup.

Let $\Gamma \le G$ be a uniform lattice. The following two notions are  variants of Definitions \ref{def:local topological rigidity} and \ref{def:local rigidity} respectively, given in terms of the Chabauty topology.

\begin{definition}
\label{def:Chabauty locally rigid}
$\Gamma$ is \textbf{Chabauty locally rigid} if the point $\Gamma \in \Sub{G}$ admits a Chabauty neighborhood consisting of uniform lattices isomorphic to $\Gamma$.
\end{definition}

\begin{definition}
\label{def:open conjugacy class}
$\Gamma $ has the  \textbf{open conjugacy class (OCC)} property  if the point $\Gamma \in \Sub{G}$ admits a Chabauty neighborhood consisting of conjugates of  $\Gamma$.
\end{definition}

The term OCC was recently coined in \cite{glasner}.

\subsubsection*{The image of a deformation}

Clearly, given any homomorphism $r : \Gamma \to G$ of some  group $\Gamma$ into a topological group $G$ the closure of $r(\Gamma)$ is a closed subgroup of $G$. 

\begin{definition}
        \label{def:the Chabauty map}
        Let $\Gamma$ be a discrete group and $G$ a locally compact group. The closure of the image of a representation is a well-defined map
        $$ \mathrm{C} : \mathcal{R}(\Gamma,G) \to \Sub{G}, \quad \mathrm{C}(r) = \overline{r(\Gamma)} .$$
\end{definition}

The map $\mathrm{C}$ relates the representation space of $\Gamma$ in $G$ to the Chabauty space of $G$. In fact, we show in \S\ref{sub:continuity of C} that $\mathrm{C}$ is continuous at the point $r_0$  corresponding to the inclusion $\Gamma \hookrightarrow G$ whenever $G$ is compactly generated and $\Gamma$ is uniform. In this sense Chabauty local rigidity is a stronger notion than  topological  local rigidity.

On the other hand, observe that if the map $\mathrm{C}$ is known to be open then  topological local rigidity implies Chabauty local rigidity, and likewise local rigidity implies the OCC property. This is the path taken in \S\ref{sub:opennes of C} towards Theorem \ref{thm:Chaubuty local rigidity}.


\section{Geometric preliminaries}
\label{sec:preliminaries}

{\it Our research on  topological local rigidity has benefited greatly from the ideas and excellent exposition of the new book \cite{cor_har} by Cornulier and de la Harpe.} 

\medskip

In this section we collect from \cite{cor_har,BrHa} a toolkit of geometric notions needed for the proof of Theorem \ref{thm:local rigidity for lc groups}. The central theme is the construction of a geometric action of  $G$ on the Rips complex,  which is a certain simplicial complex associated to a choice of a metric on $G$. The notion of compact presentation and its implication to the simple connectedness of the Rips complex is discussed. In addition we discuss  geometric actions and length spaces.

\subsection{Geometric actions}
\label{sub:geometric actions}

Let $G$ be a topological group acting on a pseudo-metric space $(X,d)$. We emphasize that the action need not be continuous.

\begin{definition}
        The action of $G$ on $X$ is:
        \label{def:properties of actions}
        \begin{enumerate}
                \item \textbf{cobounded} if there exists a bounded $F \subset X$ with $GF = X$.
                \item \textbf{locally bounded} if $\forall g \in G, x\in X$ there exists a neighborhood $g \in V$ with $Vx$ bounded.
                \item \textbf{metrically proper} if $\forall x\in X$ the subset $\{g \in G \, : \, d(gx,x) \le R\}$ is relatively compact for all $R > 0$.
                \item \textbf{geometric} if it is isometric, metrically proper, locally bounded and cobounded.
        \end{enumerate}
\end{definition}

For example, if $G$ is compactly generated and $d$ is a left invariant, proper and locally bounded\footnote{A pseudo-metric $d$ is \emph{locally bounded} if every point has a neighborhood of bounded diameter. A left invariant, proper and locally-bounded pseudo-metric on a group is called \emph{adapted} in \cite{cor_har}.} pseudo-metric then the action of $G$ on $(G,d)$ is geometric. 

\begin{prop}
        \label{pro:a cocompact subgroup is acting geometrically}
        Let $G$ be a compactly generated locally compact group and $H \le G$ be a closed cocompact subgroup. If $G$  acts on $X$ geometrically then so does $H$.
\end{prop}
\begin{proof}
        The only property that needs to be checked is coboundedness. There is a compact subset $K \subset G$ with $G =H K$ and a bounded subset $F \subset X$ with $GF = X$. Therefore $HKF = X$ and it remains to show that $KF \subset X$ is bounded.
        
        Pick a point $x_0 \in X$. Then for any $k \in K$ there is a neighborhood $k \in V_k$ with $V_k x_0$ bounded. As $K$ is compact, this implies that $Kx_0$ is bounded, and so $KF$ is bounded as well.
        
\end{proof}

\begin{lemma}
        \label{lem:condition for abstract subgroup with geometric action to be discrete}
        Let $G$ be a locally compact group admitting a geometric action on $(X,d)$. Let $H \le G$ be a subgroup considered abstractly with the discrete topology. If the restricted action of $H$ on $(X,d)$ is geometric then $H$ is a uniform lattice in $G$.
\end{lemma}
\begin{proof}
        If $H$ was not discrete in the topology induced from $G$, there would be a sequence $h_n \in H$ with $h_n\neq e, h_n \to e \in G$. Pick a point $x_0 \in X$. As the $G$-action is locally bounded, there is an open neighborhood $V \subset G$ containing $e$ such that $Vx_0$ is bounded. In particular the set $\{h_n x_0\}_{n\in \NN} \subset X$ is bounded in contrast with assumption that the action of $H$ is metrically proper.
        
        $H$ is acting coboundedly on $X$, and so there exists a bounded subset $F \subset X$ with $H F = X$. Pick a point $x_0 \in X$ and let $K = \overline{\{g \in G \: : \: gx_0 \in F\}}$. Because the action of $G$ is  metrically proper $K$ is compact. Note that $K H = G$.
\end{proof}

Recall the following coarse variant of connectedness  \cite[3.B]{cor_har}.   

\begin{definition}
\label{def:coarsely connected}
A pseudo-metric space $(X,d)$  is \textbf{coarsely connected} if $\exists c > 0$ such that $\forall x,y \in X$ there is a sequence $x = x_0, x_1, \ldots, x_n = y$ of points $x_i \in X$ with $d(x_i, x_{i+1}) < c$.
\end{definition}

The following two propositions are a first step towards the study of Chabauty neighborhoods of uniform lattices.

\begin{prop}
\label{prop:acting cocompactly is open for finitely generated}
Let $G$ be a locally compact group admitting a continuous geometric action on a coarsely connected pseudo-metric space $(X,d)$. Let $\Gamma$ be a uniform 
lattice in $G$. Then $\Gamma$ admits a Chabauty neighborhood consisting of subgroups whose action on $X$ is cobounded. 
\end{prop}


\begin{proof}
The lattice $\Gamma$ is acting on $X$ geometrically according to Proposition \ref{pro:a cocompact subgroup is acting geometrically}. In particular, 
 there is a point  $x \in X$ and a radius $R > 0$ so that $\Gamma  B_1 = X$ for $B_1 = B_{x}(R)$. We may assume without loss of generality that $X$ is $c$-coarsely connected for some $c >0$ and that $R > c $.
Denote $ B_2 = B_{x}(2R)$ and $ B_3 = B_{x}(3R)$. Let $\Sigma = \{\gamma \in \Gamma \: : \: \gamma B_1 \cap B_3 \neq \emptyset \} $ so that $B_3 \subset \Sigma B_1$. 
Since the action of $\Gamma$  on $X$ is metrically proper and the subgroup $\Gamma$ is discrete the set    $\Sigma$ is finite. 

The continuity of the action of $G$  on $X$ implies that  there is a symmetric identity neighborhood  $U \subset \Is{X}$ so that $U  B_1 \subset  B_2$. In particular 
$$ B_3 \subset \bigcup_{\sigma \in \Sigma} \sigma u_\sigma B_{2} $$
for every choice of elements $u_\sigma \in U$ with $\sigma \in \Sigma$.
We claim that the Chabauty neighborhood 
$$ \Omega = \bigcap_{\sigma \in \Sigma} \mathcal{O}_1(\sigma U) $$
of $\Gamma$ is as required.  Consider any closed subgroup  $H \in \Omega$. There are elements $h_\sigma \in H \cap \sigma U $ for every $ \sigma \in \Sigma$. We will show that in fact $H_0 B_2 = X$  so that already  the finitely generated subgroup  $ H_0 = \left<h_\sigma\right>_{\sigma \in \Sigma}$ of $H$ is   acting coboundedly on $X$.

Let $y \in X$ be any point. The $c$-coarse connectedness of $X$ allows us to find a sequence of points $x_0, x_1, \ldots, x_n$ for some $n \in \NN$, where $x_0 = x, x_n = y$ and $d(x_i,x_{i-1}) < c$ for every $ 1 \le i \le n$. Clearly $x_0 \in H_0 B_2$. Arguing by induction assume that $x_{i-1} \in h B_2$ for some $1 \le i \le n$ and $h \in H_0$. Therefore 
$$ d(x_i,h x) \le d(x_i,x_{i-1}) + d(x_{i-1}, hx) \le c + 2R \le 3R $$
In other words $x_i \in h B_3$. Note that $B_3 \subset \bigcup_{\sigma \in \Sigma} h_\sigma B_2$ so that $x_i \in h h_\sigma B_2$ for some $\sigma \in \Sigma$. In particular $x_i \in H_0 B_2$.  We conclude by mathematical induction.
\end{proof}

\begin{prop}
\label{prop:condition for a uniform lattice to admit a Chabauty neighborhood of cocompact subgroups}
Let $G$ be a locally compact group and $\Gamma$ a uniform lattice in $G$. Then $\Gamma$ admits a Chabauty neighborhood consisting of cocompact subgroups if and only if $G$ is compactly generated.
\end{prop}
Recall that any topological group admitting a uniform lattice is locally compact.
\begin{proof}
If  $G$ is compactly generated then it admits a left-invariant, continuous, proper and coarsely connected pseudo-metric    \cite[4.B.8]{cor_har}. Consider the action  of $G$ on itself equipped with such a  pseudo-metric $d$ by left multiplication. Proposition \ref{prop:acting cocompactly is open for finitely generated} implies that $\Gamma$ admits a Chabauty neighborhood consisting of cobounded subgroups. Since  $d$ is proper a cobounded subgroup is  cocompact.

Conversely, assume that $\Gamma$ admits a Chabauty neighborhood $\Omega$ consisting of cocompact subgroups. For every finite subset $F \subset \Gamma$ let $\Gamma_F$ denote the subgroup of $\Gamma$ generated by $F$. Then $(\Gamma_F)_F$ is a net of discrete subgroups of $G$ having $\Gamma$ as an accumulation point. Therefore $\Gamma_F$ belongs to $\Omega$ for some $F$. In other words $G$ admits a finitely generated uniform lattice, so that $G$ is compactly generated.
\end{proof}

\begin{remark}
\label{rem:uniformly cobounded}
The proofs of Propositions \ref{prop:acting cocompactly is open for finitely generated} and \ref{prop:condition for a uniform lattice to admit a Chabauty neighborhood of cocompact subgroups} establish that a uniform lattice $\Gamma$ in a compactly generated group $G$ admits a Chabauty neighborhood $\Omega$ and a compact subset $K \subset G$ so that every closed subgroup $H \in \Omega$ satisfies $HK = G$.
\end{remark}

\subsection{Compact presentation}
\label{sub:compact presentation}

We recall the notion of compact presentation. 


\begin{definition}
\label{def:compact presentation}
A  group $G$ is \textbf{compactly presented} if it admits a compact generating set $S$ and a presentation with relations of bounded length in $S$.
\end{definition}

A compactly presented locally compact group is boundedly presented by all its compact generating sets. See \cite[Section 8]{cor_har} for additional details.

\begin{remark}
Connected Lie groups, as well as reductive algebraic groups over local fields, are compactly presented. There are however algebraic groups over local fields which are not compactly presented, such as the group $\SL{2}{k} \ltimes k^2$ where $k$ is a local field  \cite[Section 8.A]{cor_har}.
\end{remark}

\subsection{The Rips complex}
\label{sub:the Rips complex}

Let $(X,d)$ be a pseudo-metric space and $c>0$. Consider the 2-dimensional simplicial complex $\Rips=\Rips_c^2(X,d)$ whose vertices are the points of $X$. Two distinct points $x,y \in X$ span an edge if and only if $d(x,y) \le c$. Three distinct points  $x,y,z \in X$ span a 2-simplex if and only if each pair spans an edge. 

Note that the $0$-skeleton of $\Rips_c^2(X,d)$ is naturally identified with $X$, and we frequently make this identification implicit.

\begin{definition}
The \textbf{Rips 2-complex} $\Rips_c^2(X,d)$ is the topological realization of the above simplicial complex. 
\end{definition}

$\Rips_c^2(X,d)$ is given the  \emph{weak topology}, namely each simplex is regarded with its natural Euclidean topology, and an arbitrary subset is open if and only if its intersection with each simplex is open.

\begin{definition}
\label{def:c-length} $(X,d)$ is \textbf{$\lambda$-length} if for all pairs $x,y \in X$ the distance   $d(x,y)$ equals $\inf \sum_{i=0}^{n-1}d(x_i,x_{i+1})$,  where the infimum is taken over all sequences $x = x_0,\ldots,x_n = y$ of points $x_i \in X$ with $d(x_i, x_{i+1}) < \lambda$.
\end{definition}

\begin{prop}
\label{prop:exists a pseudo-metric d for which Rips complex is simply connected}
Let $G$ be a  compactly presented locally compact group. 
Then $G$ admits a left-invariant, continuous and proper pseudo-metric $d$ so that  $\Rips_c^2(G,d)$ is connected and simply connected for all $c$ sufficiently large. 
Moreover we may take the pseudo-metric $d$ to be $\lambda$-length for some $\lambda > 0$.
\end{prop}

\begin{proof}
According to Proposition 7.B.1 and Milestone 8.A.7 of \cite{cor_har}, and as $G$ is compactly presented, the Rips complex $\Rips_c^2(G,d)$ is connected and simply connected provided that the pseudo-metric $d$ is left-invariant, proper, locally bounded and large-scale geodesic\footnote{A left-invariant, proper, locally bounded and large-scale geodesic pseudo-metric on a group is called \emph{geodesically adapted} in \cite{cor_har}.}. The existence of such a pseudo-metric which is moreover continuous follows from  the implication $(i) \Rightarrow (vi)$ in Proposition 4.B.8 of \cite{cor_har}. The fact that $d$ can be taken to be $\lambda$-length follows from the details of the proof, combined with statement $(2)$ of Lemma 4.B.7 of \cite{cor_har}.
\end{proof}

\subsubsection*{The Rips complex as a metric space}

The Rips complex $\Rips = \Rips_c^2(G,d)$ can be viewed as a quotient space obtained by gluing together a family of zero, one and two-dimensio4plices. Assume that we declare edges to have unit length and $2$-simplices to be isometric to equilateral Euclidean triangles. This defines a quotient pseudo-metric $\rho = \rho_{G,d}$ on $\Rips$ that is compatible with the weak topology on $\Rips$, see \cite[\S I.5]{BrHa}.
Metric simplicial complexes are considered in  \cite[\S I.7]{BrHa}, where a metric $\rho$ as above is called the \emph{intrinsic pseudo-metric}. 

\begin{theorem}
\label{thm:intrinsic pseudo-metric is a complete length space}
Let $d$ be a pseudo-metric on $G$. If the Rips complex $\Rips = \Rips_c^2(G,d)$  is connected for some $c > 0$ then $\Rips$ endowed with the metric $\rho_{G,d}$ is a complete geodesic space.
\end{theorem}
\begin{proof}
The simplical complex $\Rips$ has finitely many isometry types of Euclidean simplices (in fact, a unique one in every dimension).  Assuming that $\Rips$ is connected Theorem I.7.19 of \cite{BrHa} applies.
\end{proof}

It is shown in Proposition 6.C.2 of \cite{cor_har} that for a pseudo-metric $d$ as in Proposition \ref{prop:exists a pseudo-metric d for which Rips complex is simply connected},   the inclusion of $(G,d)$ into $(\Rips,\rho)$ is a quasi-isometry. This implies that the action of $G$ on $(\Rips,\rho)$ is geometric.

\begin{prop}
\label{prop:every path between vertices in the Rips complex can be approximated by an edge path}
Every two vertices $g,h \in \Rips_{(0)} = G$ are connected in $\Rips$ by a rectifiable path $p$ of length $l(p) \le 2 \cdot \rho(g,h)  $ whose image lies in the $1$-skeleton of $\Rips$.
\end{prop}
\begin{proof}

Let $g, h \in \Rips_{(0)}$ be any pair of vertices. Since $\Rips$ is a geodesic space there is a path $p' : \left[0,1\right] \to \Rips$ with $p'(0) = g, p'(1) = h$ and length $l(p') = \rho(g,h) $. 

We may without loss of generality assume that whenever $p' (\left[a,b\right]) \subset T$ for a $2$-simplex $T \subset \Rips$ with $0 \le a < b \le 1$ and $p'(a), p'(b) \in \partial T$ then $p'_{| \left[a,b\right]}$ is a geodesic segment of $T$.  Let $p$ be the path obtained from $p'$ by replacing, for any instance of $a < b$ and $T$ as above, the geodesic segment $p'_{| \left[a,b\right]}$ with a shortest path connecting $p'(a)$ and $p'(b)$ along the  boundary $\partial T$. As $T$ is isometric to an equilateral Euclidean triangle, $l(p) \le 2 \cdot l(p')$.
\end{proof}

\subsection{Length spaces and covering maps}
\label{sub:covering maps}

\begin{definition}
        \label{def:s-local isomtery}
        A map $f : X \to Y$ of metric spaces is a \textbf{local isometry} if for every $x \in X$ there is $s = s(x)$ such that the restriction of $f$ to $B_x(s)$ is an isometry onto $B_{f(x)}(s)$. 
        
        The map $f$ is an \textbf{$s$-local isometry} if $s$ can be chosen independently of $x$.\footnote{Our definition is different from the usage by several authors,
see e.g. Definition 3.4.1 of \cite{Papa}. We require a somewhat stronger
local condition, while \cite{Papa} requires in addition that a local isometry
is surjective.}
\end{definition}


\begin{lemma}
        \label{lem:a local isometry is a covering}
        Let $f : X \to Y$ be an $s$-local isometry where  $Y$ is a connected metric space and $X\ne\emptyset$. Then $f$ is a covering map.
\end{lemma}

\begin{proof}
        
        Denote by $Y'=f(X)$ the image of $f$. Since $f$ is an $s$-local isometry, $Y'$ has the property that if $y\in Y'$ then $B_y(s)\subset Y'$. This condition implies that $Y'$ is both open and closed. Since $Y$ is connected we deduce that $Y'=Y$ and $f$ is surjective.
        
        To establish the covering property, given a point $y\in Y$ consider the open neighborhoods $U_x = B_x(s/2) \subset X$ for all points $x \in f^{-1}(y)$ in the preimage of $y$. As $f$ is an $s$-local isometry it restricts to a homeomorphism on each $U_x$. Moreover $U_{x_1} \cap U_{x_2} \neq \emptyset$ implies $0 < d_X(x_1,x_2) < s$ but $f(x_1)=f(x_2) = y$ which is absurd. So the open neighborhoods $U_x$ are pairwise disjoint as required.
\end{proof}

\subsubsection*{Pullback length metric}

Recall that a metric is a  \emph{length metric} if the distance between any two points is equal to the infimum of the lengths of all rectifiable paths connecting these two points. A  space equipped with a length metric is called a \emph{length space}.
Let $f : X \to Y$ be a local homeomorphism such that $Y$ is a length space and $X$ is Hausdorff. It is possible to pullback the length metric from $Y$ to obtain a metric on $X$ that is  \emph{induced by $f$} (see Definition I.3.24 of \cite{BrHa}).

\begin{prop}
\label{prop:properties of pullback length metric}
Let $f : X \to Y$ be a local homeomorphism where $Y$ is a length space and $X$ is Hausdorff. Then the metric $d_f$ on $X$ induced by $f$ is a length metric and $f$ becomes a local isometry. Moreover $d_f$ is the unique metric on $X$ satisfying these two properties.

In addition, assume that there is $s > 0$ such that every $x\in X$ has a neighborhood $x \in U$ with $f|_{U}$ being a homeomorphism onto its image and $B_{f(x)}(s) \subset f(U)$. Then $f$ is an $s'$-local isometry for some $s' > 0$.
\begin{proof}
The first part is Proposition I.3.25 of \cite{BrHa}. The statement concerning $s$-local isometries  follows from the proof.
\end{proof}
\end{prop}

Compare Lemma \ref{lem:a local isometry is a covering} and Propositions \ref{prop:properties of pullback length metric} with the more general Proposition 3.28 of \cite{BrHa}. 

\begin{prop}
        \label{prop:a local isometry that is a homeo is an isometry}
        Let $f : X \to Y$ be a homeomorphism of length spaces that is a local isometry. Then $f$ is an isometry.
\end{prop}
\begin{proof}
        See Corollary 3.4.6 of \cite{Papa}.
\end{proof}

\section{Totally disconnected  groups}
\label{sec:totally disconnected groups}

We establish local rigidity for the isometry groups of connected and simply-connected proper length spaces assuming the action is smooth. 

\begin{definition}
	\label{def:smooth action}
	An action  of a totally disconnected topological group $G$ on a connected topological space $X$ is \textbf{smooth} if compact subsets of $X$ have open point-wise stabilizers in $G$.
\end{definition}

Any continuous action  of locally compact group on a connected locally finite graph is smooth. In addition,  the actions of  totally disconnected isometry groups on certain $\CAT$ spaces   are smooth   \cite[Theorem 6.1]{CM1}.
\begin{theorem}
\label{cor:local rigidity for td}
Let $X$ be a connected and simply connected proper length space and $G \le \Is{X}$ a closed subgroup acting cocompactly and smoothly. Then uniform lattices of $G$ are locally rigid in $\Is{X}$.
\end{theorem}
The conclusion means that a small deformation of a uniform lattice $\Gamma \le G$ is a conjugation by an element of $\Is{X}$.
This theorem, or rather its variant Theorem \ref{thm:local rigidity with totally disc factor}, is used in \S\ref{sec:cat0 spaces} as an ingredient in the proof of local rigidity for $\CAT$ groups. 

The following is yet another interesting application.

\begin{cor}
\label{cor:local topological rigidity for t.d.c.p. groups}
Let $G$ be a compactly presented totally disconnected locally compact group without non-trivial compact normal subgroups. Then   there is a locally compact group $G^\dagger$ admitting $G$ as a closed  cocompact subgroup such that  uniform lattices in $ G^\dagger$ are locally rigid.
\end{cor}

In particular, all uniform lattices in $G$ as above are   topologically locally rigid. Note that this conclusion  follows immediately from our main Theorem \ref{thm:local rigidity for lc groups}.  However  the discussion and the independent proof   given in \S\ref{sub:compactly presented totoally disconnected groups} and \S \ref{sub:uniform lattices in totally discon} is considerably simpler and yet provides an insight into the general case. 


Corollary \ref{cor:local topological rigidity for t.d.c.p. groups} can be improved in the special case where $G$ as above is a direct product of two groups and the uniform lattice $\Gamma$ has dense projections to the factors.

\begin{cor}
\label{cor:local topological rigidity for product t.d.c.p. groups}
Let $G = G_1 \times G_2$  be a compactly presented totally disconnected locally compact group without non-trivial compact normal subgroups. Let $\Gamma$ be a uniform lattice in $G$ projecting densely to both factors. Then $\Gamma$ is locally rigid in $\mathrm{Aut}(G)$.
\end{cor}

\subsection{Isometry groups with smooth action}
\label{sub:isometries groups acting smoothly}

\begin{theorem}
        \label{thm:local rigidity with totally disc factor}
        Let $ X $ be a connected and simply connected proper length space. Assume that $X$ splits as a direct product $X = X_1 \times X_2 $ of metric spaces and let $G_i \le \Is{X_i}$ be closed subgroups acting cocompactly. Assume that  $G_1$  acts smoothly. 
        
        Let $\Gamma \le G_1 \times G_2$ be a uniform lattice. Then any sufficiently small deformation of $\Gamma$ that is trivial on $G_2$ is  a conjugation by an element of $\Is{X_1}$.
\end{theorem}

This somewhat cumbersome statement involving a product decomposition is needed in the proof of local rigidity for $\CAT$ groups below.

Note that Theorem \ref{cor:local rigidity for td} is obtained as an immediate consequence of Theorem \ref{thm:local rigidity with totally disc factor} in the special case where $X \cong X_1 $, the factor $X_2$ is a reduced to a point, and a subgroup of $\Is{X} \cong \Is{X_1}$ is acting smoothly and cocompactly.


\begin{proof}[Proof of Theorem \ref{thm:local rigidity with totally disc factor}]
Let $\mathcal{R}_1(\Gamma,G)$ denote the subspace of the representation space $\mathcal{R}(\Gamma,G)$ consisting of these $r$ with $\mathrm{Proj}_2(r(\gamma)) = \mathrm{Proj}_2(\gamma)$ for all $\gamma \in \Gamma$. Let  $r_0 \in \mathcal{R}_1(\Gamma,G)$  correspond to the inclusion morphism $r_0 : \Gamma \hookrightarrow G$ and  $r \in \mathcal{R}_1(\Gamma, G)$ denote a point that should be understood as being close to $r_0$, the precise degree of closeness will be specified below. 

According to Proposition 5.B.10 of \cite{cor_har} and as $X$ is proper, $G$ is acting geometrically. In particular we may apply Proposition \ref{pro:a cocompact subgroup is acting geometrically} and find a compact subset $K \subset X$ such that $X = \Gamma K$. Denote $K_1 = \mathrm{Proj}_1(K)$ and let $O_K \le G_1$ be the point-wise stabilizer of $K_1$. Since $K_1$ is compact, smoothness implies that $O_K$ is an open subgroup of $G_1$. 

Fix some arbitrary $s > 0$ and let $N_s(K)$ denote the $s$-neighborhood of $K$, so that $N_s(K)$ is relatively compact and open. Denote
$$ \Delta_s = \{\gamma \in \Gamma \: : \: \gamma K \cap N_s(K) \neq \emptyset \}, $$
with $\Delta_s$ being finite as $\Gamma$ is discrete and its action is metrically proper.

We now construct a $\Gamma$-equivariant covering map $ f : X \to X$, where $\Gamma$ acts on the domain of $f$ via $r_0$ and on the range of $f$ via $r$, namely $$f(\gamma x) = r(\gamma)f(x), \quad \forall \gamma \in \Gamma, x \in X.$$
We propose to define $f : X \to X$ piece-wise by
$$ f_{|\gamma K} = r(\gamma) \gamma^{-1} \quad \forall \gamma \in \Gamma $$
and care needs to be taken for this to be well-defined. Note that $f$  depends on $r$ implicitly.
In any case, the above formula is equivariant in the sense that given $x \in \gamma K$ and $\gamma' \in \Gamma$ we have
$$ f_{|\gamma' \gamma K} (\gamma' x)= (r(\gamma'\gamma)\gamma^{-1}\gamma'^{-1})(\gamma' x) = r(\gamma') (r(\gamma)  \gamma^{-1}) x = r(\gamma') f_{|\gamma K}(x).$$

An ambiguity in the definition of $f$ arises whenever $\gamma K \cap \gamma' K \neq \emptyset$ for a pair of elements $\gamma, \gamma' \in \Gamma$. By the equivariance property we may assume $\gamma'=1$, that is $\gamma K \cap K \neq \emptyset $ and so  $\gamma \in \Delta_s$.
To resolve the ambiguity we require that
$$ f_{|K \cap \gamma K} = r(\gamma) \gamma^{-1}_{|K \cap \gamma K} = \mathrm{id}_{|K \cap \gamma K}. $$
This will certainly be the case given that $r(\gamma) \gamma^{-1}$ fixes the set $K$ point-wise. Note that 
$ \mathrm{Proj}_2(r(\gamma) \gamma^{-1}) = e_2 \in \Is{X_2}$ 
and so it suffices to require 
$$ \mathrm{Proj}_1(r(\gamma) \gamma^{-1}) \in O_K \Leftrightarrow \mathrm{Proj}_1 (r(\gamma)) \in O_K \mathrm{Proj}_1(\gamma). $$
This means that if we let $\Omega \subset R_1(\Gamma,G)$ be the open condition  given by
$$ \Omega = \{ r\in R_1(\Gamma, G) \: : \: \mathrm{Proj}_1 (r(\gamma)) \in O_K \mathrm{Proj}_1(\gamma),~  \forall \gamma \in \Delta_s \}, $$
the function $f$ is well-defined as long as  $r \in \Omega$.

We next ensure that $f$ is  a local isometry. We define an open neighborhood $\Omega_{s}$ of $r_0$ in $R_1(\Gamma,G)$ so that for every $f \in \Omega_{s}$ the function $f$ is an $s$-local isometry. Let
$$\Omega_s = \{ r\in R_1(\Gamma, G) \: : \: \mathrm{Proj}_1 (r(\gamma)) \in \mathrm{Proj}_1(\gamma) O_K \; \forall \gamma \in \Delta_s \} $$
and it is easy to see that for every $r \in \Omega_s$ and $\gamma \in \Delta_s$ the map $r(\gamma) \gamma^{-1}$ fixes point-wise the set $\gamma K$. In particular $f_{|N_s(K)} = \mathrm{id}$ for all such $r \in \Omega_s$.

Given a point $x \in X$ we may by equivariance assume that $x \in K$ and so $B_s(x) \subset N_s(K)$. In particular $f_{|B_s(x)} = \mathrm{id}$ and $f$ is certainly a $s$-local isometry at $x$. 

We have shown that $f : X \to X$ is a $\Gamma$-equivariant $s$-local isometry. Therefore $f$ is a covering map by Lemma \ref{lem:a local isometry is a covering}.  The connectedness and simple connectedness  of $X$ implies that $f$ is an homeomorphism. A local isometry which is a homeomorphism between length spaces is in fact an isometry, see Proposition \ref{prop:a local isometry that is a homeo is an isometry}.

To conclude we may regard $f \in  \Is{X}$ and the equivariance formula implies that $f \Gamma f^{-1} = r(\Gamma)$  as long as $r \in \Omega \cap \Omega_s$ for some $s > 0$. 
Observe that $\mathrm{Proj}_2(f) = e_2 \in \Is{X_2}$ and so in fact $f \in \Is{X_1}$.
        
\end{proof}

\subsection{The Cayley-Abels graph}
\label{sub:compactly presented totoally disconnected groups}

In order to  apply Theorem \ref{cor:local rigidity for td} towards Corollary \ref{cor:local topological rigidity for t.d.c.p. groups} we need to construct  a suitable space $X$ associated with a totally disconnected compactly presented group $G$. With this aim in mind we recall the notion of a Cayley--Abels graph; see e.g. \cite{moller,wesolek} or  \cite[2.E.10]{cor_har}.

\begin{prop}
\label{prop:cayley-abels graph}
Let $G$ be  a compactly generated  totally disconnected locally compact group. There is a locally finite connected graph $X$ on which $G$ acts transitively, continuously and with compact open stabilizers. 
\end{prop}
\begin{proof}
By the  van-Dantzig theorem there is a compact open subgroup $U \subset G$. We refer to Proposition 2.E.9 of \cite{cor_har} for the construction of the required graph $X$. In particular, the vertices of $X$ are taken to be the cosets $G/U$. 
\end{proof}

Any graph satisfying the conditions of Proposition \ref{prop:cayley-abels graph} is called a \emph{Cayley--Abels} graph\footnote{If $G$ is a countable discrete group then a Cayley-Abels graph is simply a Cayley graph with respect to some generating set.}
 for $G$. It is clear that the action of $G$ on a corresponding Cayley--Abels graph  is geometric. In particular all such graphs for $G$ are quasi-isometric to $G$ and therefore to each other, as follows from Theorem 4.C.5 of \cite{cor_har} (this fact is also proved in \cite{moller}).

\begin{prop}
\label{prop:rips complex for t.d. is simply connected}
Let $G$ be a compactly presented totally disconnected locally compact group and let $X$ be a Cayley--Abels graph for $G$. Then the Rips complex $\Rips_c^2(X^{(0)})$ constructed with respect to the graph metric on the vertices of $X$ is simply-connected for all $c$ sufficiently large.
\end{prop}

Note that the Rips complex $\Rips_c^2(X^{(0)})$ constructed on the vertices of $X$ is a locally finite $2$-dimensional simplicial complex.

\begin{proof}[Proof of Proposition \ref{prop:rips complex for t.d. is simply connected}]
In terms of the machinery developed in \cite{cor_har} the proof is transparent. Namely, $G$ is large-scale equivalent (i.e. quasi-isometric) to $X$. The fact that $G$ is compactly presented is equivalent to $G$ being coarsely simply connected\footnote{The notion of \emph{coarse simple connectedness} is discussed in  \cite[6.A]{cor_har}.} with respect to some geodesically adapted metric; \cite[8.A.3]{cor_har}. Hence $X$ is coarsely simply connected, which implies that $\Rips_c^2(X)$ is simply connected for all $c$ sufficiently large by  \cite[6.C.6]{cor_har}. 

Alternatively, a hands on proof can be obtained by recalling that the vertices of $X$ correspond to $G/U$ for some compact open subgroup $U$ and verifying that the proof of  \cite[7.B.1]{cor_har} goes through in this case as well.
\end{proof}

\subsection{Local rigidity of uniform lattices in  totally disconnected groups}
\label{sub:uniform lattices in totally discon}

\begin{proof}[Proof of Corollary \ref{cor:local topological rigidity for t.d.c.p. groups}]
Let $G$ be a compactly presented totally disconnected locally compact group without non-trivial compact normal subgroups. Let $X$ be a Cayley--Abels graph for $G$ and $\Rips = \Rips_c^2(X)$ the corresponding Rips complex for $c$ sufficiently large so that $\Rips$ is connected and simply connected, as in Proposition \ref{prop:rips complex for t.d. is simply connected}. Denote $G^\dagger = \Is{\Rips}$ and let $\alpha : G \to G^\dagger$ be the natural  homomorphism \cite[5.B]{cor_har}. The group $G^\dagger$ is locally compact, totally disconnected and compactly presented, and the homomorphism $\alpha$ is   continuous, has compact kernel, and its image  is closed and cocompact in $G^\dagger$. Our assumptions imply that $\alpha$ is injective so that  $G$ can be identified with the subgroup $\alpha(G)$ of $G^\dagger$.

It is clear from the definition of a Cayley--Abels graph that the action of $G^\dagger$ on $\Rips$ is smooth. Moreover $\Rips$ is a length space, as in Theorem \ref{thm:intrinsic pseudo-metric is a complete length space}. Being a locally finite simplicial complex with a single isometry type of simplices in each dimension, $\Rips$ is proper. The group $G^\dagger$ is acting transitively on the vertices of $\Rips$ and hence cocompactly on $\Rips$.
Having verified all the required conditions, local rigidity of uniform lattices in $G^\dagger$ follows immediately from Theorem \ref{cor:local rigidity for td}.
\end{proof}


\begin{remark}
\label{remark:uniform tree lattices}
 Uniform tree lattices are locally rigid ---
 let $T$ be a locally finite tree and $G$  a cocompact group of automorphisms of $T$. Every uniform lattice $\Gamma$ in $G$ admits a  free  Schottky subgroup $\Delta$ of finite index \cite{lubotzky1991lattices}. A Schottky subgroup is defined by an open condition and it is clear that $\Delta$ is  topologically locally rigid. Since the centralizer of $\Gamma$ in $G$ is trivial, our Lemma \ref{lem:local rigidity and finite index} implies  that $\Gamma$ is  topologically locally rigid as well.
%
\end{remark}
\begin{remark}
\label{remark:another proof of local rigidity}
The previous remark leads to a  different and possibly simpler proof of Corollary \ref{cor:local topological rigidity for t.d.c.p. groups}.
Let $G$ be a compactly generated totally disconnected locally compact group. There is a locally compact group  $\widetilde{G}$ admitting a quotient map onto $G$ with discrete kernel, such that $\widetilde{G}$ is acting by automorphisms on the universal covering tree  of some Cayley--Abels graph for $G$. Uniform lattices in $\tilde{G}$ are  topologically locally rigid by Remark \ref{remark:uniform tree lattices}.  Topological local rigidity for uniform lattices in $G$ follows relying on the methods of  \S \ref{sec:locally compact case}.
\end{remark}


\begin{proof}[Proof of Corollary \ref{cor:local topological rigidity for product t.d.c.p. groups}]
Let $G = G_1 \times G_2$ be a product of two compactly presented totally disconnected locally compact groups without non-trivial compact normal subgroups and $\Gamma$ a uniform lattice in $G$ projecting densely to both factors. Let $X$ be the product of the two Rips complexes $\Rips_1$ and $\Rips_2$ associated  to the two factors  by Proposition \ref{prop:rips complex for t.d. is simply connected} and as in the proof of Corollary \ref{cor:local topological rigidity for t.d.c.p. groups}. We consider $X$ with the product $L^2$ metric. The de Rham decomposition theorem of Foertsch--Lytchak \cite{foertsch2008rham} implies that $G^\dagger = \Is{\Rips_1} \times \Is{\Rips_2}$  is an open finite index subgroup of $\Is{X}$. In particular the uniform lattice $\Gamma$ is locally rigid in $G^\dagger$ by Theorem  \ref{cor:local rigidity for td}. We conclude relying on Proposition \ref{pro:locally rigid in and dense projections}.
  \end{proof}


\section{Compactly presented groups}
\label{sec:compactly presented}

In this section we prove Theorem \ref{thm:local rigidity for lc groups} with the additional assumption of compact presentation. Let $G$ be a compactly presented group and $\Gamma \le G$ a uniform lattice.  We  show that a small deformation $r(\Gamma) $ of $\Gamma$ is a uniform lattice isomorphic to it.

\subsubsection*{Overview}

The strategy of the proof  is as follows. Consider the Rips complex $\Rips$ associated to $G$.  Then $\Rips$ is connected and simply connected  and $G$ is acting on $\Rips$ geometrically.
The main ingredient of the proof is the construction of  a connected space $Y$ and a covering map $ f : Y \to \Rips$ depending on the deformation $r$.  Such a map $f$ must of course be a homeomorphism.

We consider  certain open subsets of  $\Rips$ called \emph{bubbles}. The space $Y$ is constructed  by gluing together bubbles indexed by elements $\gamma \in \Gamma$. The map $f $ is  naturally associated to this gluing and is a local homeomorphism. This allows  us to define a pull-back length metric on $Y$. We then promote $f$ to an $s$-local isometry using a certain ``Lebesgue number" argument with respect to the cover of $Y$ by bubbles.

There is an obvious action of $\Gamma$ on $Y$ by homeomorphisms, coming from the indexing of the bubbles. The map $f$ becomes $\Gamma$-equivariant with respect to this action, with $\Gamma$ acting on $\Rips$ via $r$. Finally, we show that $\Gamma$ is  acting on $Y$ geometrically with respect to the induced length metric and deduce from this Theorem \ref{thm:local rigidity for lc groups}.

\subsubsection*{Standing notations}

Recall that $\Gamma$ is a uniform lattice in the compactly presented group $G$.  Let $r_0 \in \mathcal{R}(\Gamma,G)$ correspond to the  inclusion and $r \in \mathcal{R}(\Gamma,G)$ be a point close to $r_0$.

Let $\Rips = \Rips_c^2(G,d)$ be the associated Rips complex where $d$ is some left invariant, continuous and proper  pseudo-metric on $G$, see Proposition \ref{prop:exists a pseudo-metric d for which Rips complex is simply connected}. We assume that $d$ is $\lambda$-length and that $c > \lambda$ is sufficiently large so that $\Rips$ is connected and simply-connected. Moreover $\rho$ is the intrinsic metric on $\Rips$, and $\bar{\rho}$ will denote the induced length metric on the space $Y$.

We will  use the notation $U = \mathcal{B}(e,D)$, where $\mathcal{B}$ denotes a bubble as in Definition \ref{def:bubbles}. Here $ D > 0$ is a fixed radius. An additional radius $L$ (the ``Lebesuge number") is introduced in Lemma \ref{lem:lebesgue number} and it is assumed that $ L > c > \lambda$ and $ D > D(L)$. See Lemma \ref{lem:lebesgue number} for the meaning of the parameter $D(L)$.

Finally, $\Sigma$ will denote a certain generating set for $\Gamma$ depending on the choice of the bubble $U$.

\subsection{Bubbles}

We consider certain open subsets of $\Rips$ that share many of the properties of metric balls.

\begin{definition}
        \label{def:spans}
        Let $\Rips$ be a simplicial complex and $A \subset \Rips_{(0)}$ a subset of its $0$-skeleton. The \textbf{span $\mathcal{S}(A) \subset \Rips$} of $A$ is the union of those  simplices in $\Rips$ all of whose vertices belong to $A$.
\end{definition}

\begin{definition}
        \label{def:bubbles}
        Let $(\Rips,\rho) = \Rips_c^2(G,d)$ be the Rips complex of $(G,d)$ with its associated  metric $\rho$. Let $0 < \varepsilon < \sqrt{3}/4$ be fixed.
        The \textbf{bubble} at a vertex $g \in G = \Rips_{(0)}$ of radius $D > 0$ is
        $$ \mathcal{B}(g,D,\varepsilon) = \mathcal{N}_\varepsilon\left( \mathcal{S}\left(B\right)\right) = \{ x \in \Rips \: : \: \rho(x,\mathcal{S}(B)) < \varepsilon \}$$
        where $$B = B_{(G,d)}\left(g,D\right).$$
\end{definition}

The constant $ 0 < \varepsilon < \sqrt{3}/4$ will be kept fixed in what follows and its exact value will not affect the arguments in any way. For this reason we will sometimes suppress $\varepsilon$ from the  notation for $\mathcal{B}$.

We collect basic properties of bubbles.

\begin{prop}
        \label{prop:properties of a bubble}
        Let $\mathcal{B} = \mathcal{B}(x,D,\varepsilon)$ be a bubble.
        \begin{enumerate}
                \item \label{item:B is open} $\mathcal{B}$ is open in $\Rips$.
                \item \label{item:B is connected} If $d$ is $\lambda$-length, $\Rips = \Rips_c^2(G,d)$ and $D > c > \lambda$ then $\mathcal{B}$ is connected.
                \item \label{item:action on bubbles} Given $g\in G$ we have $g \mathcal{B} = \mathcal{B}(gx,D,\varepsilon)$.
                \item \label{item:a middle point in contained in the bubble} $x \in \mathcal{B} = \mathcal{B}(x,D,\varepsilon)$.
                
        \end{enumerate}
        Let $\mathcal{B}' = \mathcal{B}(x',D',\varepsilon)$ be an additional bubble.
        \begin{enumerate}[resume]
                \item \label{item:inclusion of smaller bubbles} If $d(x,x') + D' < D$ then $\mathcal{B}' \subset \mathcal{B}$.
        \end{enumerate}   
        Let $\mathcal{B}_i = \mathcal{B}(x_i,D_i,\varepsilon)$ be a collection of bubbles for $i\in I$ with $\abs{I}<\infty$.
        \begin{enumerate}[resume]
                \item \label{item:intersection of bubbles}  If $\cap_{i\in I}\mathcal{B}_i \neq \emptyset$ then this intersection contains a vertex of $\Rips$.
                
        \end{enumerate}
\end{prop}
\begin{proof}
        Item (\ref{item:B is open}) is obvious.
        
        To establish (\ref{item:B is connected}), consider a bubble at the vertex $x = g \in G$. Theorem \ref{thm:intrinsic pseudo-metric is a complete length space} implies that $\rho-$balls  are path connected in $\Rips$, and so it suffices to show that $\mathcal{S}(B)$ is path connected. Any point of $\mathcal{S}(B)$ can in turn be connected by a segment to one of the vertices in $B$.
It remains to show that any vertex $h \in B$ can be connected to $g$ within $\mathcal{S}(B)$. As $d(g,h) < D$ and $d$ is $\lambda$-length there is a sequence of points $x_0 = g, x_1, \ldots, x_n = h$ with $d(x_i,x_{i+1}) < \lambda < c$ and $d(g,x_i) < D$. The corresponding path in $\Rips$ connects $g$ to $h$ and is contained in $\mathcal{S}(B)$ as required.

        Statement (\ref{item:action on bubbles}) follows immediately from Definition \ref{def:bubbles} and the left $G$-invariance of the pseudo-metric $d$ and metric $\rho$, and statements (\ref{item:a middle point in contained in the bubble}) and (\ref{item:inclusion of smaller bubbles}) are clear.

        To observe (\ref{item:intersection of bubbles}) we make an auxiliary construction. Let $\Rips'$ be the simplicial complex obtained by subdividing every $2$-cell of $\Rips$ in the obvious way into four  $2$-cells isometric to an equilateral triangle of side length $1/2$. Every edge of $\Rips$ is subdivided into two new edges in $\Rips'$ of length $1/2$. Clearly the underlying metric spaces of $\Rips$ and $\Rips'$ are isometric. Given a subcomplex $ Q \subset \Rips$ let $\mathcal{N}'(Q)  \subset \Rips'$ denote the union of the interiors of those cells of $\Rips'$ whose closure has non-empty intersection with $Q$. Since $\varepsilon < \sqrt{3} / 4$ we see that
        $$ \mathcal{B}_i = \mathcal{N}_\varepsilon(\mathcal{S}(B_i)) \subset \mathcal{N}'(\mathcal{S}(B_i)) $$
        where $B_i = B_{(G,d)}(x_i,D_i)$. However, note that because of the definition of $\Rips'$ 
         $$\mathcal{N}'(Q_1 \cap Q_2) = \mathcal{N}'(Q_1) \cap \mathcal{N}'(Q_2)$$
     for every pair of subcomplexes $Q_1, Q_2 \subset \Rips$. Moreover $\mathcal{S}(B_i) \cap \mathcal{S}(B_j) = \mathcal{S}(B_i \cap B_j)$. Statement (\ref{item:intersection of bubbles}) follows.
%
%
%

%
%
\end{proof}

Bubbles provide us with a system of nicely behaved open subsets of $\Rips$. However, the ``center points" of bubbles belong by definition to the $0$-skeleton $G$ of $\Rips$. This will not pose any serious difficulty since any point of $\Rips$ is at most a unit distance from $G$.

\begin{prop}
        \label{prop:bubbles eventually contain balls}
        Let $(\Rips,\rho) = \Rips_c^2(G,d)$ be a Rips complex associated to $G$. Any $\rho$-bounded subset is contained in the bubble $\mathcal{B}(e,D)$ for all $D$ sufficiently large.
\end{prop}
\begin{proof}
        Let $F \subset \Rips$ be a $\rho$-bounded subset. Let  $V$ be the union of all vertices belonging to simplices of $\Rips$ which intersect $F$ non-trivially. Then $V$ is bounded as well and in particular $V \subset B_\rho(e,N) \subset \Rips$ for some $N >0$ sufficiently large. Regarded as a subset of $G$, this implies $V \subset B_d(e,cN)$. We obtain
        $ F \subset \mathcal{S}(V) \subset \mathcal{B}(e,D) $
        for all $D > cN$.
\end{proof}

Essentially, Proposition \ref{prop:bubbles eventually contain balls} follows immediately from the fact that $(G,d)$ and $(\Rips,\rho)$  are quasi-isometric.

\subsection{The nerve of a deformation.}

\label{sub: the nerve}

\begin{definition}
\label{def:nerve}
Let $\Delta \subset \Gamma$ be any subset, $U \subset \Rips$ an open set and $r \in \mathcal{R}(\Gamma,G)$. The \textbf{nerve} $N(\Delta, U, r)$ consists of all subsets $ A \subset \Delta$ such that  $\bigcap_{\sigma \in A} r(\sigma) U $
is non-empty.
\end{definition}

For example, the condition $\{\sigma_1, \sigma_2\} \in N(\Delta, U, r)$ implies that $r(\sigma_1) u_1 = r(\sigma_2) u_2$ with $u_1, u_2$ ranging over certain open subsets of $U$.


\begin{prop}
\label{prop:close representations have same nerve}
Let $\Delta \subset \Gamma$ be any finite subset and $U = \mathcal{B}(g, D) \subset \Rips$ a bubble at a vertex $g \in G = \Rips_{(0)}$. Then, denoting $U' = \mathcal{B}(g, D - \alpha)$ the equality
$$N(\Delta, U, r_0) = N(\Delta, U', r_0) = N(\Delta, U', r)$$ 
holds for every $r \in \mathcal{R}$ sufficiently close to $r_0$ and $\alpha > 0$ sufficiently small.
\end{prop}
\begin{proof}
By definition $A = \{\delta_1, \ldots, \delta_m\} \in N(\Delta, U, r)$ if and only if the bubbles $r(\delta_1)U, \ldots, r(\delta_m)U$ intersect non-trivially. By Items (\ref{item:action on bubbles}) and (\ref{item:intersection of bubbles}) of Proposition \ref{prop:properties of a bubble} this happens if and only if there exists $g_A \in G$ that satisfies
$$ 
 d(r(\delta) g, g_A) < D \quad \forall \delta \in A. 
$$
Since the pseudo-metric $d$ is continuous and $\Delta$ is finite, for $r$ sufficiently close to $r_0$ and $\alpha$ sufficiently small the conditions
$$N(\Delta, U, r_0) = N(\Delta, U', r_0) \subset N(\Delta, U', r)$$ 
are satisfied. Note that $N(\Delta, U, r_0) \supset N(\Delta, U', r_0) $ follows trivially from the definitions.

To obtain the opposite inclusion, first choose $\alpha > 0$ sufficiently small as required above. It is then clear from Item (\ref{item:inclusion of smaller bubbles}) of Proposition \ref{prop:properties of a bubble} that $r(\delta)U' \subset \delta U$ for all $\delta \in \Delta$ and $r$ sufficiently close to $r_0$. In particular $N(\Delta, U, r_0) \supset N(\Delta, U', r)$ and equality holds throughout.
\end{proof}

We remark that for our purposes it will only be relevant to consider elements of the nerve of size at most three.

\begin{remark}
In general, let $G$ be a group acting by homeomorphisms on a connected path-connected space $X$ and admitting an open path-connected subset $U \subset G$ with $GU = X$. Then one can derive a  presentation for $G$ from knowledge of the nerve $N(G,U)$. For details see \cite{macbeath} or \cite[I.8.10]{BrHa}.
\end{remark}


\subsection{The space $Y$} 
\label{sub: the space Y}

Consider the action of $\Gamma$ on the Rips complex $\Rips$. By Proposition \ref{pro:a cocompact subgroup is acting geometrically} and the remark following Theorem \ref{thm:intrinsic pseudo-metric is a complete length space} this action is geometric, and in particular cobounded. By Proposition \ref{prop:bubbles eventually contain balls} there is some $D > 0$ such that $\Gamma U = \Rips$ where $U$ is the bubble
$$ U = \mathcal{B}(e,D) $$
and $e \in G$ is regarded as a vertex of $\Rips$. The bubble $U$ is open and connected according to Proposition \ref{prop:properties of a bubble}. In addition, denote
$$ \Sigma = \{\gamma \in \Gamma \: : \: \gamma U \cap U \neq \emptyset \} $$
so that $\Sigma$ is a finite symmetric generating set for $\Gamma$, containing the identity element.

\begin{remark}
Using our notation for nerves, $\Sigma$ can be  defined alternatively as consisting of those $\sigma \in \Gamma$ such that $\{1,\sigma\} \in N(\Gamma, U, r_0)$.
\end{remark}

\subsubsection*{Construction of $Y$ as a quotient space}

Consider the space $\Gamma \times U$ with the product topology, $\Gamma$ being discrete. We introduce a certain equivalence relation $\mathcal{E}_r$ on $\Gamma \times U$. 

\begin{definition}
\label{def:the equivalence relation Er}
$\mathcal{E}_r$ is the equivalence relation generated by
$$ (\gamma, u) \sim _{\mathcal{E}_r} (\gamma', u') $$
if $\gamma' = \gamma \sigma$ for some $\sigma \in \Sigma$ and $r(\gamma) u = r(\gamma') u'$.
\end{definition}

We emphasize that the element $\sigma$ in the above definition must belong to $\Sigma$.

\begin{definition}
\label{def:the space Y and map f}
Let $Y$ be the quotient space $(\Gamma \times U)/ \mathcal{E}_r $. 
Let $\left[\gamma,u\right] \in Y$ denote the $\mathcal{E}_r$-class of the point $(\gamma,u)$. For a subset $V \subset U$ let $\left[\gamma,V \right] = \{\left[\gamma, v\right]:v\in V \}$.
\end{definition}

The map 
$$f' : \Gamma \times U \to \Rips, \quad (\gamma, u) \mapsto r(\gamma) u $$ 
factors through a well-defined map $$f : Y \to \Rips.$$ 
Note that the space $Y$, as well as the map $f$, implicitly depend on the choice of $r \in \mathcal{\Rips}$, that has not yet been specified but should be thought of as close to $r_0$. 

The following  propositions establish several useful properties of the space $Y$ and the map $f : Y \to \Rips$.

\begin{prop}
\label{prop: Y is connceted}
The space $Y$ is connected, assuming that $r \in \mathcal{R}$ is sufficiently close to $r_0$.
\end{prop}
\begin{proof}
Since $U$ is path-connected so is $\left[\gamma, U\right]$ for every $\gamma \in \Gamma$; see Proposition \ref{prop:properties of a bubble}.(\ref{item:B is connected}). For $ r = r_0$ the proposition  follows from the definition of $\Sigma$ and the fact that $\Gamma = \gener{\Sigma}$. The same argument shows that $Y$ is connected as long as $r$ is sufficiently close to $r_0$ so that Proposition \ref{prop:close representations have same nerve} applies with the data $\Delta=\Sigma$ and  $U$.
\end{proof}

\begin{prop}
\label{prop:basic properties of Y and f}
The $\mathcal{E}_r$-quotient map is a local homeomorphism onto $Y$, and in particular $Y$ is Hausdorff.  The map $f$ is a local homeomorphism, namely $f$ restricts to a homeomorphism on $\left[\gamma,U\right], \forall\gamma\in \Gamma$.
\end{prop}
\begin{proof}
Let $\pi : \Gamma \times U \to Y$ denote the quotient map with respect to $\mathcal{E}_r$ and denote $f' = f \circ \pi$. Note that 
$$ u_1 \neq u_2 \Rightarrow \left[\gamma, u_1 \right] \neq \left[\gamma,u_2 \right]\quad \forall \gamma \in \Gamma \; \forall u_1, u_2 \in U $$
which implies that the restriction of $\pi$ to each $\{\gamma\} \times U$ is a homeomorphism and $\pi$ is a local homeomorphism. Similarly, Proposition \ref{prop:properties of a bubble}.(\ref{item:B is open}) shows that $f'$ is a local homeomorphism.

The facts that $\Rips$ is Hausdorff and $\pi$ a local homeomorphism imply that $Y$ is Hausdorff. Since $\pi$ and $f'$ are both local homeomorphisms and $\pi$ is surjective, it follows that $f$ is a local homeomorphism as well.
\end{proof}

\begin{prop}
\label{prop:on Er-closures}
The $\mathcal{E}_r$-closure of $\{e\} \times U$ is contained in $\Sigma \times U$, as long as $r \in \mathcal{R}$ is sufficiently close to $r_0$.
\end{prop}


%

\begin{proof}

Assume that $D > 0$ is a priori chosen so that $U' = \mathcal{B}(x_0, D - \alpha)$ satisfies $\Rips = \Gamma U'$, for $\alpha > 0$ sufficiently small. Using Proposition \ref{prop:close representations have same nerve}, up to replacing  $U$ by $U'$ if necessary and without loss of generality, we may assume that  $N(\Sigma^2, U, r) = N(\Sigma^2, U, r_0)$ for all $r$ sufficiently close to $r_0$. Denote this nerve by $N$. In particular, an element $\gamma \in \Sigma^2$ satisfies $\{1,\gamma\} \in N$ if and only if $\gamma \in \Sigma$.

Consider an $\mathcal{E}_r$-equivalent pair of points $(e,u), (\gamma, u') \in \Gamma \times U$. We claim that in fact $\gamma \in \Sigma$.  By Definition \ref{def:the equivalence relation Er} we may write $\gamma = \sigma_1 \cdots \sigma_n$ where $\sigma_i \in \Sigma$ (possibly with repetitions), and there are points $u_1, \ldots, u_n \in U$ with $u_n = u'$ and
$$
 u = r(\sigma_1) \cdots r(\sigma_m) u_m, \quad ~\text{for}~ 1 \le m \le n.
$$

We use induction to show that $\sigma_1 \cdots \sigma_m \in \Sigma$, for $1 \le m \le n$. The base case $m = 1$ is clear. Next assume that $\sigma = \sigma_1 \cdots \sigma_{m-1} \in \Sigma$. Since
$$ u \in U \cap r(\sigma_1 \cdots \sigma_{m-1}) r(\sigma_m) U = U \cap r(\sigma)r(\sigma_m) U$$
this intersection is non-empty. Namely $\{1,\sigma\sigma_m\} \in N$ and by the above $\sigma\sigma_m \in \Sigma$. 

The last induction step $m = n$ shows that $\gamma \in \Sigma$ as required.
\end{proof}

\subsubsection*{The action of $\Gamma$ on $Y$}

$\Gamma$ has a natural action on $\Gamma \times U$, which descends to a $\Gamma$-action on $Y$ by homeomorphisms. Note that  $f$ is $\Gamma$-equivariant with respect to this action on $Y$ and the action via $r$ on $X$, namely
$$ f( \gamma \left[\gamma', u\right] ) = r(\gamma) r (\gamma') u = r(\gamma \gamma') u, \quad \forall u \in U, \gamma, \gamma' \in \Gamma. $$


%
%

\subsection{$f$ is a covering map}

We have constructed a local homeomorphism $f : Y \to \Rips$, depending on $r\in \mathcal{R}$. 
The Rips complex $\Rips$ is a length space by Theorem \ref{thm:intrinsic pseudo-metric is a complete length space} and $Y$ is Hausdorff by Proposition \ref{prop:basic properties of Y and f}. Therefore we may use $f$ to induce a length metric on $Y$.

\medskip

\emph{Let $\bar{\rho}$ be the pullback length metric on $Y$ induced by $f : Y \to (\Rips,\rho)$.}

\medskip

This makes $f$ a local isometry according to Proposition \ref{prop:properties of pullback length metric}. However, we would like to apply the additional clause of that proposition to deduce that $f$ is actually an $s$-local isometry for some $s>0$. Towards this aim we make the following key observation, which is somewhat reminiscent of the existence of a Lebesgue number\footnote{Unlike the classical Lebesgue number, no compactness argument is used in Lemma \ref{lem:lebesgue number}.}.

\begin{lemma}
\label{lem:lebesgue number}
Consider  $\Rips=\Rips_c^2(G,d)$ and $r \in \mathcal R$ sufficiently close to $r_0$. For every $L > 0$ a radius $D = D(L) > 0$ may be chosen sufficiently large, so that for any vertex $x \in f(Y) \cap G \subset \Rips_{(0)}$ the inclusions
$$\mathcal{B}(x,L) \subset f(\left[\gamma, U\right]) \subset f(Y) \subset \Rips$$ 
hold for some $\gamma \in \Gamma$.
\end{lemma}

Here $D$ is taken to be radius of the bubble $ U = \mathcal{B}(e,D)$  used in Definition \ref{def:the space Y and map f} to construct the space $Y$. In particular,  the choice of the generating set $\Sigma$ depends on $D$. It is understood that in our discussion we first specify $L$  and then $U$ and $\Sigma$ are chosen accordingly.

\begin{proof}
Let $L > 0$ be arbitrary and choose $D = D(L)$ sufficiently large so that 
$$ U' = \mathcal{B}\left(e,D - L - \eta \right) $$
satisfies $ \Rips = \Gamma U' $ for some $\eta > 0$.
Say  $x \in r(\gamma) U = f(\left[\gamma,U \right])\subset f(Y)$ with $\gamma \in \Gamma$. As $f$ is $\Gamma$-equivariant, we may assume by applying $\gamma^{-1}$  that $x\in U$ and it then suffices to show that $\mathcal{B}(x,L) \subset r(\sigma) U$ for some $\sigma \in \Sigma$.

By the assumption on $U'$ we know that $ x \in \gamma U'$ for some $\gamma \in \Gamma$. So $x \in U \cap \gamma U' \subset U \cap \gamma U$ which implies $\gamma = \sigma \in \Sigma$.
Applying Proposition \ref{prop:properties of a bubble} we obtain
$$\mathcal{B}(x,L) \subset \sigma \mathcal{B}(e,D-\eta) = \mathcal{B}(r_0(\sigma),D-\eta). $$
If $r$ is sufficiently close to $r_0$ so that $d(\sigma, r(\sigma)) < \eta$ for all $\sigma \in \Sigma$ the above conclusion  can be modified to read $$
\mathcal{B}(x,L) \subset \mathcal{B}(r(\sigma),D) =  r(\sigma) U = f(\left[\sigma,U \right]) \subset f(Y)$$
as required.
\end{proof}



\begin{cor}
\label{cor:f is a covering map}
        Assume that $\Rips = \Rips_c^2(G,d)$ and $U = \mathcal{B}(e,D)$ for $c,D > 0$ sufficiently large, and that $r$ is sufficiently close to $r_0$ in $\mathcal{R}(\Gamma,G)$. Then $f$ is a 
       homeomorphism.
\end{cor}

\begin{proof}

Fix a "Lebesgue constant" $L >  c $ and choose $D > 0$ as provided by Lemma \ref{lem:lebesgue number}. This in turn determines the bubble $U = \mathcal{B}(e,D)$ and the generating set $\Sigma \subset \Gamma$.

It has already been established that $f : Y \to \Rips$ is a 
local isometry. Lemma \ref{lem:lebesgue number} provides the additional assumption needed in Proposition \ref{prop:properties of pullback length metric} to deduce that $f$ is an $s$-local isometry for some $s > 0$. Lemma \ref{lem:a local isometry is a covering}  shows that $f$ is a covering map.

With $Y$ being connected by Proposition \ref{prop: Y is connceted} and $\Rips$ being connected and simply-connected the fact that $f$ is a homeomorphism follows\footnote{It is for this argument that we crucially rely on $G$ being compactly presented.}. We assume throughout that $r$ is sufficiently close to $r_0$ as required.
\end{proof}

\subsection{A geometric action of $\Gamma$ on $Y$}
\label{sub:metric on Y and the action is geometric}


We have established  that $f:Y \to \Rips$ is a $\Gamma$-equivariant homeomorphism. According to Proposition \ref{prop:a local isometry that is a homeo is an isometry} the map $f$ is in fact an isometry. In other words, $(Y,\bar{\rho})$ is simply an isometric copy of the Rips complex with an action of $\Gamma$ via the representation $r$.


\begin{prop}
\label{pro:Gamma is acting geometric on Y}
 $\Gamma$ is acting on $(Y,\bar{\rho})$  geometrically.
\end{prop}
\begin{proof}
The action is clearly by isometries. Moreover, as $Y = \Gamma \left[e,U\right]$ it is cobounded. As $\Gamma$ is discrete the action is locally bounded. 

It remains to verify that the action of $\Gamma$ is metrically proper. Let $d_\Sigma$ be the word metric on $\Gamma$. It suffices to show that $d_\Sigma(\gamma_n,e) \to \infty$ implies $\bar{\rho}\left(\left[\gamma_n,e\right],\left[e,e\right]\right) \to \infty$, with $e$ being regarded both as an element of $\Gamma$ and as a vertex in the $0$-skeleton of $U \subset \Rips$.
Let $\gamma \in \Gamma$ be an arbitrary element and consider the two points $$y_e = \left[e,e\right], \quad y_\gamma = \left[\gamma,e\right]$$ of $Y$. According to Proposition \ref{prop:every path between vertices in the Rips complex can be approximated by an edge path} there is a path $p : \left[0,1\right] \to Y$ with 
$$
 p(0) = y_e, p(1) = y_\gamma, ~\text{ of length}~l(p) \le 2 \cdot \bar{\rho}(y_e,y_\gamma) 
$$ 
so that the image of $p$ lies in the $1$-skeleton $Y_{(1)}$. 

We claim that 
$$l(p) \ge d_\Sigma(\gamma,e) - 2$$ 
recalling that the length $l(p)$ equals the number of edges along $p$. To prove this claim let $v_0, v_1, \ldots, v_{l(p)}$ denote the vertices of $Y_{(0)}$ that appear along the path $p$. For each $0 \le m \le l(p)$ choose $\delta_m \in \Gamma$ as provided by Lemma \ref{lem:lebesgue number} so that $\mathcal{B}(v_m, L) \subset \left[\delta_m,U\right]$. Since $d(v_{m-1},v_m) < c$ and $L$ was chosen so that $L > c$, this implies $v_{m} \in \mathcal{B}(v_{m-1}, L)$ for all $1 \le m \le l(p)$. In particular, by Proposition \ref{prop:on Er-closures} and assuming $r$ is sufficiently close to $r_0$ we see that $\delta_{m} \in \delta_{m-1} \Sigma$ for $1 \le m \le  l(p)$. But the same reasoning shows $\delta_0 \in \Sigma$ and $ \gamma  \in \delta_m \Sigma$ and the required estimate  on $l(p)$ follows.

Putting together this claim  and the lower bound on $\bar{\rho}(y_e, y_\gamma)$ in terms of $l(p)$ we conclude that the action of  $\Gamma$ on $(Y,\bar{\rho})$ is indeed metrically proper.

\end{proof}

\begin{prop}
\label{prop:Gamma is acting faithfully on the Rips complex}
$\Gamma$ is acting on $(Y,\bar{\rho})$ faithfully.
\end{prop}
\begin{proof}
Given an element $\gamma \in \Gamma$ and a point $\left[\gamma', x\right] \in Y$ for some $\gamma' \in \Gamma$ and $x \in \mathcal{B}$, recall that
$ \gamma\left[\gamma',x \right] = \left[\gamma \gamma', x \right]$. Therefore an element $\gamma$ which is not in  $\Sigma$ certainly acts non-trivially by Proposition \ref{prop:on Er-closures}. Consider the remaining case of an element $\gamma = \sigma \in \Sigma$. Then
$$ \sigma\left[e, e \right] = \left[e, e\right] \; \Leftrightarrow \; r(\sigma) = e  $$
and so $\sigma$ is acting non-trivially on $Y$ as long as $r(\sigma) \neq e \in G$.
\end{proof}

As is commonly the case throughout this section, we implicitly assume in both propositions that $r : \Gamma \to G$ is sufficiently close to $r_0$.

\subsection{Proof assuming compact presentation}
\label{sub:proof assuming comapct presentation}

Putting together the results of this section we complete the proof of Theorem \ref{thm:local rigidity for lc groups} in the compactly presented case, showing that a small deformation $r$ of $\Gamma$ is injective, discrete and uniform.

%

\begin{proof}[Proof of Theorem \ref{thm:local rigidity for lc groups} assuming $G$ is compactly presented]

Recall the notations introduced in the beginning of this section. In particular, fix $L > c$ arbitrary and let $D = D(L)$ be as required for Lemma \ref{lem:lebesgue number} to hold.  Consider the bubble $$U = \mathcal{B}(e,D) \subset \Rips$$ 
so that $\Rips = \Gamma U$. Let $\Sigma \subset \Gamma$ be the finite generating set corresponding to $U$. 


We are now in a situation to construct the space $Y$ and the map $f : Y \to \Rips$ depending on $r$, as in \S\ref{sub: the space Y}. Corollary \ref{cor:f is a covering map}  and Proposition \ref{prop:a local isometry that is a homeo is an isometry} imply that $f$ is a $\Gamma$-equivariant isometry with $\Gamma$ acting on $\Rips$ via $r$. It was established in Propositions \ref{pro:Gamma is acting geometric on Y} and \ref{prop:Gamma is acting faithfully on the Rips complex} that the $\Gamma$-action on $Y$ equipped with the induced length  metric is geometric and faithful.

We conclude that $r:\Gamma\to r(\Gamma)$ is an isomorphism. Therefore the subgroup $r(\Gamma) \le G$ considered abstractly with the discrete topology is acting geometrically on $\Rips$. In particular $r(\Gamma)$ is in fact a uniform lattice in $G$ according to Lemma \ref{lem:condition for abstract subgroup with geometric action to be discrete}. 


\end{proof}

We extract from the above proof two corollaries that will become useful in the study of Chabauty local rigidity in \S\ref{sec:the Chabauty space of closed subgroups}. As before, let $\Gamma$ be a uniform lattice in the compactly presented group $G$.

\begin{cor}
\label{cor:fundamental domain remains a fundamental domain after a deformation}
There is a compact subset $K \subset G$ so that $r(\Gamma) K = G$ for every deformation $r$ sufficiently close to the inclusion mapping.
\end{cor}
\begin{proof}
The fact that the $\Gamma$-equivariant map $f : Y\to X$ is an isometry implies in particular that $r(\Gamma)U = \Rips$. We may take $K$ to be the  closed ball $\overline{B}_{(G,d)}(e,D)$, or equivalently the closure in $G$ of the vertex set of the bubble $U = \mathcal{B}(e,D)$. Note that $K$ is compact since the pseudo-metric $d$ is proper.
\end{proof}

\begin{cor}
\label{cor:avoidance of compact set}
 Let $Q \subset G$ be a compact subset so that $\Gamma \cap Q = \emptyset$. Then $r(\Gamma) \cap Q = \emptyset$ provided that the deformation $r$ is sufficiently small.
\end{cor}

\begin{proof}
We proceed as in the above proof of  topological local rigidity for the compactly presented case. Choose the  radius $D > 0$ to be sufficiently large for the purposes of that proof and so that moreover $Q \subset \mathcal{B} = \mathcal{B}(e,D)$ where $\mathcal{B}$ denotes a bubble at $e$ of radius $D$, see Proposition \ref{prop:bubbles eventually contain balls}. 
     
Note that $r(\gamma) \in Q$ implies in particular  $\mathcal{B} \cap r(\gamma)\mathcal{B} \neq \emptyset$.   This last condition implies that $\gamma\in \Sigma$ for some finite subset $\Sigma$ of $\Gamma$ and for every sufficiently small deformation $r$, see Proposition \ref{prop:on Er-closures}. The requirement that $r(\sigma) \notin Q$ for every $\sigma \in \Sigma$ is an open neighborhood of the inclusion in $\mathcal{R}(\Gamma,G)$.
%
%
\end{proof}

\section{From compactly generated to compactly presented groups}
\label{sec:locally compact case}

We now discuss  Theorem \ref{thm:local rigidity for lc groups} in the  general case where $\Gamma \le G$ is a uniform lattice and $G$ is compactly generated.
The main idea is to reduce the question back to the compactly presented situation, using the following result. 

\begin{prop}
\label{prop:on existence of simply connected covers}
Let $G$ be a  compactly generated locally compact group. Then there is a compactly presented locally compact group $\tilde{G}$ admitting a discrete normal subgroup $N \lhd \tilde{G}$ so that $\tilde{G}/N \cong G$ as topological groups.
\end{prop}
\begin{proof}
See Proposition 8.A.13 and Corollary 8.A.15 of \cite{cor_har}
\end{proof}

Consider the compactly presented group $\tilde{G}$ given by Proposition \ref{prop:on existence of simply connected covers} and let 
$$ p : \tilde{G} \to G, \quad \ker p = N \lhd \tilde{G}$$
be the associated homomorphism and its kernel. For a subgroup $H \le G$ denote $\tilde{H} = p^{-1}(H) \le \tilde{G}$. 

In this situation we have two representation spaces $\mathcal{R} = \mathcal{R}(\Gamma, G)$ and $\tilde{\mathcal{R}} = \mathcal{R}(\tilde{\Gamma}, \tilde{G})$ with the points $r_0 \in \mathcal{R}$ and $\tilde{r}_0 \in \tilde{\mathcal{R}}$ corresponding to the inclusions $\Gamma \le G, \tilde{\Gamma} \le \tilde{G}$ respectively.

\begin{lemma}
        \label{lemma:pulling and pushing uniform lattices}
        If $\Gamma \le G$ is a uniform lattice so is $\tilde{\Gamma} \le \tilde{G}$. Conversely if $\tilde{\Gamma} \le \tilde{G}$ is a uniform lattice containing $N$, then $\Gamma = p(\tilde{\Gamma}) \le G$ is a uniform lattice.
\end{lemma}
\begin{proof}
        Let $\Gamma \le G$ be a uniform lattice with $G = \Gamma K$ for some compact set $K \subset G$. There is an open set $U \subset G$ such that $U \cap \Gamma = \{e\}$ and so $p^{-1}(U) \cap \tilde{\Gamma} = N$. Let $V \subset \tilde{G}$ be an open set such that $V \cap N = \{\tilde{e}\} \subset \tilde{G}$. Then $\tilde{\Gamma} \cap (p^{-1}(U)\cap V) = \{\tilde{e}\}$ and hence $\tilde{\Gamma} \le \tilde{G}$ is discrete. Using the fact that $\tilde{G}$ is locally compact and by Lemma 2.C.9 of \cite{cor_har} there exists a compact subset $\tilde{K} \subset \tilde{G}$ with $p(\tilde{K}) = K$. It follows that $\tilde{\Gamma} \tilde{K} = \tilde{G}$ and $\tilde{\Gamma}$ is a uniform lattice.
        
        Conversely let $N \le \tilde{\Gamma} \le \tilde{G}$ be a uniform lattice with $\tilde{G} = \tilde{\Gamma} \tilde{K}$ for $\tilde{K} \subset \tilde{G}$ compact. The discreteness of $N = \ker p$ implies that $p$ is a local homeomorphism (see e.g. \cite{pontry}, p. 81). Take $U \subset \tilde{G}$ open and sufficiently small so that $\tilde{\Gamma} \cap U =\{\tilde{e}\}$ and $p(U) \cong U$. Since $N\le \ti\Gamma$ we also have $p(U) \cap \Gamma = \{e\}$, hence $\Gamma$ is discrete. Since $G = \Gamma p(\tilde{K})$, $\Gamma$ is a uniform lattice. 
        \end{proof}

\begin{definition}
\label{def: the space R_N and map Rp}The space $\tilde{\mathcal{R}}$ admits a closed subspace $\tilde{\mathcal{R}}_N$ given by
$$ \tilde{\mathcal{R}}_N = \{\tilde{r} : \tilde{\Gamma} \to \tilde{G} \: : \: \tilde{r}(n) = \tilde{r}_0(n) = n,~\forall n \in N\}. $$
Associated to the quotient map $p$ there is a well-defined map
$$ p_* : \tilde{\mathcal{R}}_N \to \mathcal{R}, \quad p_* r (\gamma) = p\circ r(\tilde{\gamma}) \quad \forall \gamma \in \Gamma $$
where $\tilde{\gamma} \in \tilde{\Gamma}$ in any element satisfying $p(\tilde{\gamma}) = \gamma$.
\end{definition}

Note that clearly $\tilde{r}_0 \in \tilde{\mathcal{R}}_N$. The map $p_*$ is not injective in general.

\begin{prop}
\label{prop:the map Rp is continuous and open}
The map $p_* : \tilde{\mathcal{R}}_N \to \mathcal{R}$ is a local homeomorphism.
\end{prop}
\begin{proof}
Recall that $\mathcal{R}(\Gamma,G)$ and $\tilde{\mathcal{R}}_N$ may  be identified with closed subsets of the two spaces $G^\Gamma$ and $\tilde{G}^{\tilde{\Gamma}}$, respectively, equipped with the product topology. Let $\gamma \in \Gamma$ be any element and $\tilde{\gamma} \in \tilde{\Gamma}$ be such that $p(\tilde{\gamma}) = \gamma$. The equation $(p_* r) (\gamma) = p(r(\tilde{\gamma}))$ implies that $p_*$ is  both continuous and open.

Let $\tilde{\Sigma}$ be a finite generating set for $\tilde{\Gamma}$ and  $\tilde{U} \subset \tilde{G}$  a symmetric open identity neighborhood so that $\tilde{U}^2 \cap N = \{e\}$. For every  point $\tilde{r} \in \tilde{\mathcal{R}}_N$ consider the open neighborhood  
 $\tilde{r} \in \tilde{\Omega} \subset \tilde{\mathcal{R}}_N$   determined by the condition that $\tilde{s}(\tilde{\sigma}) \in \tilde{r}(\tilde{\sigma})\tilde{U}$ for every $\tilde{\sigma} \in \tilde{\Sigma}$ and every $\tilde{s} \in \tilde{\Omega}$. This clearly implies that $p_*$ is injective on $\tilde{\Omega}$. We conclude that $p_*$ is a local homeomorphism.
\end{proof}

\begin{proof}[Proof of Theorem \ref{thm:local rigidity for lc groups} for compactly generated groups]
Let $G$ be a compactly generated group and $\Gamma \le G$ a uniform lattice.
Let $\ti G,p,N$ be as in Proposition \ref{prop:on existence of simply connected covers}.
Lemma \ref{lemma:pulling and pushing uniform lattices} implies that $\tilde{\Gamma} = p^{-1}(\Gamma)$ is a uniform lattice in $\tilde{G}$.

Let $\mathcal{R}$ and $\tilde{\mathcal{R}}$ be the associated representation spaces, as in the above discussion. Applying Theorem \ref{thm:local rigidity for lc groups} for compactly presented groups (see \S\ref{sub:proof assuming comapct presentation}) we deduce the existence of an open neighborhood $\tilde{r}_0 \in \tilde{\Omega} \subset \tilde{\mathcal{R}}$ that satisfies the conclusion of the theorem for the uniform lattice $\tilde{\Gamma} \le \tilde{G}$. Next consider the subset 
$$ \Omega = p_* (\tilde{\Omega}\cap \tilde{\mathcal{R}}_N) \subset \mathcal{R} $$
and it follows from Proposition \ref{prop:the map Rp is continuous and open} that $\Omega$ is an open neighborhood of $r_0$ in $\mathcal{R}$.

We claim that Theorem \ref{thm:local rigidity for lc groups} holds for $\Gamma \le G$ with respect to the neighborhood $\Omega \subset \mathcal{R}$. To see this, consider a representation $r \in \Omega$ and let $\tilde{r} \in \tilde{\Omega} \cap \tilde{\mathcal{R}}_N$ be such that $p_*\tilde{r} = r$. Then $\tilde{r}$ is injective and $\tilde{r}(\tilde{\Gamma})$ is a uniform lattice in $\tilde{G}$ containing $N$ and satisfying $\tilde{\Gamma} \cong \tilde{r}(\tilde{\Gamma})$. From Lemma \ref{lemma:pulling and pushing uniform lattices} it follows that $r(\Gamma) = p(\tilde{r}(\tilde{\Gamma}))$ is a uniform lattice in $G$. The injectivity of $\tilde{r}$ and the fact that $\tilde{r}(n) = n$ for all elements $n \in N \le \tilde{\Gamma}$ implies that $r$ is injective as well.
\end{proof}



\section{Chabauty local rigidity}
\label{sec:the Chabauty space of closed subgroups}

We study uniform lattices regarded as points in the Chabauty space, establishing Chabauty local rigidity and the open conjugacy class property.
Recall the \emph{image closure}  map associated to a uniform lattice $\Gamma$ in a topological group $G$, namely  
$$\mathrm{C} : \mathcal{R}(\Gamma,G) \to \Sub{G}, \quad \mathcal{R}(\Gamma,G) \ni r \mapsto \overline{r(\Gamma)} \in \Sub{G}$$
 Our current goal is the following result.


\begin{theorem}
	\label{thm:Chaubuty local rigidity}
	Let $G$ be a topological group without non-trivial compact normal subgroups. Let $\Gamma$ be a uniform lattice in $G$ and  $r_0 \in \mathcal{R}(\Gamma,G)$  the inclusion mapping. 
	\begin{enumerate}
\item If $G$ is compactly generated then $\Gamma$ has a Chabauty neighborhood consisting of uniform lattices.
\item If $G$ is compactly presented then $\mathrm{C}$ is a local homeomorphism at $r_0$ and $\Gamma$ has a Chabauty neighborhood consisting of isomorphic uniform lattices.
\end{enumerate}
\end{theorem}

Theorem \ref{thm:Chabauty local rigidity for compacty presented groups} of the introduction is essentially a restatement of    Theorem \ref{thm:Chaubuty local rigidity}.



%



\subsection{The map $\mathrm{C}$ is continuous  for compactly generated groups}
\label{sub:continuity of C}




\begin{prop}
        \label{prop:C is continuous at a uniform lattice}
        Let $\Gamma$ be a uniform lattice in the compactly generated group $G$ and $r_0$ the inclusion mapping. Then $\mathrm{C}$ is continuous at $r_0$.
\end{prop}

In proving the above proposition we may assume without loss of generality that $G$ is compactly presented. To see this, let $\tilde{G}$ be a compactly presented group admitting a quotient map $p : \tilde{G} \to G$ with discrete kernel $N = \ker p$, as in Proposition \ref{prop:on existence of simply connected covers}. Consider the commutative square
\[
\begin{tikzcd}
\tilde{\mathcal{R}}_N \arrow{r}{\tilde{\mathrm{C}}} \arrow[swap]{d}{p_*} & \Subb{\tilde{G}}{N} \arrow{d}{p_*} \\
\mathcal{R}(\Gamma,G)  \arrow{r}{\mathrm{C}} & \Sub{G}
\end{tikzcd}
\]
The representation space $\tilde{\mathcal{R}}_N$ and the map $p_*$ along the left vertical arrow are described in Definition \ref{def: the space R_N and map Rp}. Here $\Subb{\tilde{G}}{N}$ is the space of closed subgroups of $\tilde{G}$ containing $N$ and the map $p_*$ from that space to $\Sub{G}$ is the projection modulo $N$. The map $p_* : \tilde{\mathcal{R}}_N \to \mathcal{R}$ is a local homeomorphism  by Proposition \ref{prop:the map Rp is continuous and open}.
Now observe that the continuity of $\mathrm{C}$ follows from that of the map $\tilde{\mathrm{C}}$.


\begin{proof}[Proof of Proposition \ref{prop:C is continuous at a uniform lattice}]
        Recall the sub-basis sets generating the Chabauty topology given in Definition \ref{def:the Chabauty topology}. It suffices to verify that the preimage under $\mathrm{C}$ of every such sub-basis set  containing $\Gamma$ contains a neighbourhood of $r_0$. 
        This is straightforward for sets of the form $\mathrm{C}^{-1}\left(\mathcal{O}_1(U)\right)$ with  $U$ being open in $G$. For sets of the form $\mathrm{C}^{-1}\left(\mathcal{O}_2(K)\right)$ with  $K$ being a compact subset of $ G$ we rely on Corollary \ref{cor:avoidance of compact set}.

\end{proof}

\subsection{A Chabauty neighborhood consisting of discrete subgroups} 

As a preliminary step towards proving Theorem \ref{thm:Chaubuty local rigidity} we study  Chabauty neighborhoods of uniform lattices in topological groups without non-trivial compact normal subgroups. The main result is Proposition \ref{prop:uniform lattice has neighborhood of uniform lattices} below.



\begin{lemma}
\label{lem:on small normal subgroups}
Let $G$ be a locally compact group and  $\Gamma \le G $ a discrete subgroup.   Let $U$ be  a symmetric relatively compact identity neighborhood in $G$ with $\Gamma \cap U = \{e\}$. Then  every Chabauty neighborhood $\Omega$ of $\Gamma$ has a Chabauty sub-neighborhood $\Omega_U$ with $\Gamma \in \Omega_U \subset \Omega$ such that  every closed subgroup $H \in \Omega_U$ satisfies
\begin{enumerate}
	\item $H \cap U$ is a subgroup of $G$, \label{item:intersection is a subgroup} and
	\item There is a closed subgroup $L \le H$ so that $L \in \Omega$ and $L$ normalizes $H \cap U$.\label{item:small normal subgroup}

\end{enumerate}

%
\end{lemma}

\begin{proof}

We may assume, up to passing to a smaller neighborhood $\Omega$ if needed, that there are open relatively compact  subsets $V_1, \ldots, V_n \subset G$ and a compact subset $K \subset G$ so that
$$ \Omega = \bigcap_{i=1}^n \mathcal{O}_1(V_i) \cap \mathcal{O}_2(K) $$
in terms of the notations of Definition \ref{def:the Chabauty topology}. Let
%
$$ V = \bigcap_{i = 1}^n \bigcap _{v \in V_i} U^{v} $$
so that $V$ is an identity neighbourhood in $G$ by \cite[Theorem 2.4]{montgomery}. In addition let $W$ be an open symmetric identity neighbourhood satisfying $W ^2 \subset U \cap V$.

Consider the Chabauty neighborhood $\Omega_U$, given by
$$ \Omega_U = \Omega  \cap \mathcal{O}_2(\bar{U}\setminus{W}).$$
Let $H \in \Omega_U$ be any closed subgroup of $G$. Observe that the facts
$ H \cap W = H \cap U $,  $W^2 \subset U$ and $W = W^{-1}$ taken together imply that the intersection $H \cap U$ is a subgroup, establishing (\ref{item:intersection is a subgroup}).
Choose elements $h_i \in H \cap V_i$ for every $i = 1,\ldots, n$. 
It remains to show that the closed subgroup $ L = \overline{\left<h_1,\ldots,h_n\right>} \le H$ satisfies  (\ref{item:small normal subgroup}), that is $L$ normalizes $H \cap U$. 
Note that
$$ (H\cap U)^{h_i} = (H\cap V)^{h_i} = H \cap V^{h_i} \subset H \cap U $$
and so $H \cap U$ is indeed normalized by $\left<h_1,\ldots,h_n\right>$ and hence by $L$, as required.
%
%
%
%
\end{proof}


%

\begin{prop}
\label{prop:uniform lattice has neighborhood of uniform lattices}
Let $G$ be a compactly generated group without non-trivial compact normal subgroup. Let  $\Gamma$ be a uniform lattice in $G$ satisfying $\Gamma \cap U = \{e\}$ for some relatively compact symmetric open identity neighborhood $U$ in $G$. Then $\Gamma$ admits a Chabauty neighborhood  consisting of uniform lattices $\Delta$ with $\Delta \cap U = \{e\}$.
\end{prop}
\begin{proof}
The uniform lattice $\Gamma$ admits a Chabauty neighborhood $\Omega$ consisting of cocompact subgroups according to Proposition \ref{prop:condition for a uniform lattice to admit a Chabauty neighborhood of cocompact subgroups}. Let $U$ be a relatively compact symmetric open identity neighborhood in $G$ so that $\Gamma \cap U = \{e\}$. Relying on Lemma \ref{lem:on small normal subgroups} we find a smaller Chabauty neighborhood $\Omega_U$ of $\Gamma$ with $\Omega_U \subset \Omega$ and satisfying Conditions $(1)$ and $(2)$ of Lemma \ref{lem:on small normal subgroups}.  

Let $H$ be any closed subgroup in $\Omega_U$. Since $\Omega \subset  \Omega_U$ it follows that $H$ is cocompact in $G$. It remains to show that the intersection $V = H \cap U$ is the trivial subgroup. Note that $V$ is a compact subgroup of $G$. Moreover there is a closed subgroup $  L \le H$ so that $L$ normalizes $V$ and so that $L$ is cocompact in $G$.

Let $v \in V$ be an arbitrary element. Since $V$ is normalized by $L$ the set $v^L = \{ v^l \: : \: l \in L \} $ is relatively compact. Additionally since $L$ is cocompact in $G$ it follows that the entire conjugacy class $v^G = \{ v^g \: : \: g \in G\}$ of $v$ is relatively compact.
Let $J $ be the closed subgroup of $G$ generated by  $v^G$. In particular $J$ in normal in $G$ and it admits a dense subset consisting of elements with a relatively compact conjugacy class in $G$. Observe that $J$ is compactly generated since the compact set  $\overline{v^G}$ generates a dense subgroup of $J$. 

It follows from \cite[Lemma 7]{kramarev1971topological} and \cite{ushakov1963topological} that $J$ admits a unique largest compact subgroup $K \nrm J$ and that $J / K$ is abelian without non-trivial compact subgroups. This implies that $K$ is characteristic in $J$ and therefore $K \nrm G$. By the assumption $K$ is trivial and $J$ is abelian without non-trivial compact subgroups.
Recall that $ v \in V \cap J$ and therefore $v$ generates a compact subgroup of $J$. However by the above the element $v$ must be trivial. We conclude that $V = H \cap U = \{e\}$ as required.
\end{proof}

\subsection{Proof of Theorem \ref{thm:Chaubuty local rigidity}} 
\label{sub:opennes of C}

Let $G$ be a compactly generated group without non-trivial compact normal subgroups and $\Gamma \le G$ a uniform lattice.   It follows from Proposition \ref{prop:uniform lattice has neighborhood of uniform lattices} that $\Gamma$ admits a Chabauty neighborhood $\Omega$ consisting of uniform lattices $\Delta$ satisfying $\Delta \cap U = \{e\}$ for some relatively compact symmetric open identity neighborhood $U$ in $G$. This completes the first part of Theorem  \ref{thm:Chaubuty local rigidity}.

Assume  that $G$ is moreover compactly presented. 
Recall that the continuity of $\mathrm{C}$ at $r_0$ was given in Proposition \ref{prop:C is continuous at a uniform lattice} above. It remains  to verify that $\mathrm{C}$ is open at $r_0$. In other words,  we would like to show that every $H \in \Omega $  is actually equal to the image of some small deformation of $\Gamma$.  It would then follows from Theorem \ref{thm:local rigidity for lc groups} that $\Gamma$ admits a Chabauty neighborhood consisting of \emph{isomorphic} uniform lattices.

\subsubsection*{The map $\mathrm{C}$ is open at the point $r_0$}

We  rely on the fact that $\Gamma$ is finitely presented \cite[1.D.4]{cor_har}.  Let $\Sigma$ be a finite generating set for $\Gamma$ and let $V$ be an identity neighborhood in $G$. Up do passing to a smaller Chabauty neighborhood $\Omega$, we may choose elements $h_\sigma \in H \cap V \sigma  $  for every $\sigma \in \Sigma$. We propose to construct a map $r_H : \Gamma \to H $ by defining it on generators\footnote{The notation $r_H$ is a slight abuse of notation, since the map $r_H$ depends not only on the group $H$ but also on the choices of the elements $h_\sigma$.}, declaring $r_H(\sigma) = h_\sigma$ for all $\sigma \in \Sigma$.
In order that $r_H$ extends to a well-defined homomorphism, it must send all the relations of $\Gamma$ to the identity. Since $\Gamma$ is finitely presented, and in view of the fact that $H\cap U=\{e\}$, the map $r_H$ is indeed  well-defined provided that the neighborhood $V$ is sufficiently small. Therefore $r_H(\Gamma) \le H$ and it remains to show that  $r_H(\Gamma)$ is  equal to $H$, possibly up to passing once more to a smaller Chabauty neighborhood.


There is a  symmetric compact subset $K \subset G$ so that $r_H(\Gamma) K = G$ for every closed subgroup $H$ in $\Omega$ provided that the neighborhood $V$ is sufficiently small,  see Corollary \ref{cor:fundamental domain remains a fundamental domain after a deformation}. In particular every element $h \in H$ can be written as $h = r_H(\gamma) k_h$ for some elements $\gamma \in \Gamma$ and $k_h \in K \cap H$. The equality   $r_H(\Gamma) = H$  will follow as soon as we show that $H \cap K \subset r_H(\Gamma)$ whenever $r_H$ is a sufficiently small deformation.

Let $W$ be a symmetric open identity neighborhood satisfying $W^2 \subset U$. Let $\Delta$ be the finite collection of all elements $\delta \in \Gamma$ such that $K \cap W \delta \neq \emptyset$. Consider the compact subset $ Q = K  \setminus W \Delta$. 
Every closed subgroup $H$ of $G$ belonging to the Chabauty neighborhood $\Omega_1$
$$\Omega_1 = \Omega \cap \bigcap_{\delta \in \Delta} \mathcal{O}_1(W\delta ) \cap \mathcal{O}_2(Q) \subset \Sub{G} $$
intersects each translate $W\delta$ exactly once. Such a  subgroup $H$ therefore satisfies
$$H \cap K  \subset H \cap W\Delta \subset r_H(\Gamma)$$
and in particular  $r_H(\Gamma) = H$ by the above discussion. This completes the proof of the fact that $\mathrm{C}$ is open at $r_0$ as well as of Theorem \ref{thm:Chaubuty local rigidity}.
\qed



%
%

  \section{Local rigidity of $\CAT$ groups}
\label{sec:cat0 spaces}

We turn to the question of local rigidity for uniform lattices in the isometry groups of $\CAT$ spaces. Our goal is to prove Theorem \ref{thm:local rigidity of uniform in general cat0-intro} of the introduction.

%

In this section we rely to a large extent on the work of Caprace and Monod concerning the structure theory and discrete subgroups of  isometry groups of $\CAT$ spaces \cite{CM1,CM2}. For additional general information concerning these spaces we refer the reader to the books \cite{BrHa} and \cite{Papa}.

\subsection{$\CAT$ groups}
\label{sub:cat0 lattices and their properties}

We now collect several relevant facts about $\CAT$ spaces and their lattices to be used below. These are recalled for the reader's convenience.

Let $X$ be a proper $\CAT$ space with a cocompact isometry group $\Is{X}$. Then $\Is{X}$ is locally compact, second countable and is acting on $X$ geometrically. Moreover as $X$ is geodesic and simply-connected $\Is{X}$ is compactly presented; see 5.B.5, 5.B.10, 6.A.10 and 8.A.8 of \cite{cor_har}.

\begin{definition}
        \label{def:a CAT0 lattice}
        A \emph{$\CAT$ lattice} is a lattice in $\Is{X}$. A uniform $\CAT$ lattice is called a \emph{$\CAT$ group}.
\end{definition}

The mere fact that the space $X$ admits a $\CAT$ group already has some significant implications. Assume that $X$ is proper and \emph{geodesically complete}, that is every geodesic segment can be extended to a bi-infinite geodesic line, and that $G \le \Is{X}$ is a closed subgroup acting cocompactly  and admitting a uniform lattice $\Gamma \le G$. Assume moreover that $X$ has no Euclidean factors. Then:

\begin{itemize}
        \item The boundary $\partial X$ has finite geometric dimension (Theorem C, \cite{Kl}).
        \item $\Gamma$ is acting minimally and without a fixed point in $\partial X$ (geometric Borel density --- Theorem 3.14, \cite{CM2}). In particular, $G$ acts minimally and admits no fixed point at infinity as well.
                \item The centralizer of a  lattice in $\Is{X}$ is trivial \cite[2.7]{CM2}.
                        
\end{itemize}



%
%
%

%
%
%
%

\subsection{Irreducibility and splitting}
\label{sub:irreducibility and splitting}

We recall two notions of irreducibility developed in Section 4 of \cite{CM2}. In particular, we show that a $\CAT$-group virtually decomposes into a product of abstractly irreducible groups.

\begin{definition}
        \label{def:abstractly irreducible}
        A discrete group is \textbf{abstractly irreducible} if no finite index  subgroup splits nontrivially as a direct product.
\end{definition}

\begin{definition}
        \label{def:irreducible lattice}
        A lattice $\Gamma \le G_1 \times \cdots \times G_n$ in a product of locally compact groups is \textbf{irreducible} if its projection to every proper subproduct of the $G_i$'s is dense and each $G_i$ is non-discrete.
\end{definition}

In this terminology, an abstractly irreducible CAT(0) lattice becomes an irreducible lattice in the product of the closures of its projections (Remark 4.1, \cite{CM2}).

In the case of a uniform lattice $\Gamma \le G$ such that $G$ is a cocompact subgroup of $\Is{X}$ and $X$ a proper $\CAT$ space, these two notions of irreducibility turn out to be related. Namely, $\Gamma$ is abstractly irreducible if and only if for every finite index subgroup $\Gamma' \le \Gamma$ and a decomposition $X = X_1 \times X_2$, the projections of $\Gamma'$ to both $\Is{X_i}$ are non-discrete (see Theorem 4.2 of \cite{CM2}).

\begin{remark}
\label{rem:having non-discrete deformations is preserved under small deformations}
As abstract irreducibility is an algebraic condition, 
it follows from Theorem \ref{thm:local rigidity for lc groups} that in the setting of the previous paragraph the property of having non-discrete projections is stable under small deformations.
\end{remark}

\begin{lemma}
        \label{lem:existence of splitting}
        Let $X$ be a proper $\CAT$ space without Euclidean factors with $\Is{X}$ acting cocompactly and $\Gamma \le G$ a uniform lattice.
        Then there is a finite index  normal subgroup $\Gamma' \nrm \Gamma$ and product decompositions
        $$ \Gamma' = \Gamma_1 \times \cdots \times \Gamma_k,  \quad X = X_1 \times \cdots \times X_k, \quad  $$ 
        such that the $\Gamma_i \le \Is{X_i}$ are abstractly irreducible uniform lattices.
\end{lemma}

We provide a proof for this lemma as we could not find one in the literature.

\begin{proof}
        Note that every finite index subgroup $\Gamma_0$ of $\Gamma$ is a uniform lattice and so acts minimally and without a fixed point at infinity (Theorem 3.14, \cite{CM2}). By the splitting theorem (see Theorem 9 and in particular Corollary 10 in \cite{Mo})  every direct product decomposition of such a $\Gamma_0$ into finitely many factors gives rise to an associated product decomposition of $X$ with the same number of factors.
        
        However, there is an a priori canonical maximal isometric splitting of the space $X$ into a product of finitely many irreducible spaces (see Theorem 1.9 of \cite{CM1}). Therefore the process of passing to a further finite index subgroup of $\Gamma$ and writing an abstractly reducible group as a direct product must terminate after finitely many iterations. This gives a finite index subgroup $\Gamma' \le \Gamma$ that decomposes as the direct product of $k$ abstractly irreducible factors $\Gamma_i$. The associated decomposition of $X$ is obtained as above, and it is clear that every $\Gamma_i$ is a uniform lattice in $\Is{X_i}$.
        
        It remains to show that $\Gamma'$ can be taken to be normal in $\Gamma$. Indeed $\Gamma$ preserves the decomposition of $X$ upon possibly permuting isometric factors, as do all elements of $\Is{X}$. Therefore every conjugate $(\Gamma')^\gamma$ with $\gamma \in \Gamma$ decomposes as a direct product in a manner similar to $\Gamma'$. The intersection of such conjugates with $\gamma$ running over finitely many coset representatives for $\Gamma / \Gamma'$ is as required.
\end{proof}

\subsection{Totally disconnected isometry groups}

We first deal with the special case of Theorem \ref{thm:local rigidity of uniform in general cat0-intro} where $G$ is totally disconnected.  This relies on the results of \S\ref{sec:totally disconnected groups} and in particular Theorem \ref{cor:local rigidity for td}.


\begin{thm}
        \label{thm:local rigidity purely td cat0}
        Let $X$ be a proper geodesically complete $\CAT$ space with $G \le \Is{X}$ totally disconnected and acting cocompactly.  Then uniform lattices in $G$ are locally rigid in $\Is{X}$.
\end{thm}
Namely, a small deformation within $G$ of a uniform lattice $\Gamma \le G$ is a conjugation by an element of $\Is{X}$.

\begin{proof}
        According to Theorem 6.1 of \cite{CM1} and as $G$ is totally disconnected, $G$ is acting on $X$ smoothly in the sense of Definition \ref{def:smooth action}. As $X$ is proper and $G$ is acting cocompactly by the assumptions, Theorem \ref{cor:local rigidity for td} applies in the present situation and uniform lattices are indeed locally rigid in $\Is{X}$.

\end{proof}

\subsection{Proof of Theorem \ref{thm:local rigidity of uniform in general cat0-intro}}

We proceed to establish the local rigidity of lattices in proper geodesically complete $\CAT$ spaces without Euclidean factors. The proof will rely on the following important result, established by Caprace and Monod.  Namely, the $\CAT$ space $X$ admits a splitting
$$ X = M \times Y $$
where $M$ is a symmetric space of non-compact type. Correspondingly, $\Is{X}$ has a finite index open characteristic subgroup $\Is{X}^*$ that decomposes as
$$ \Is{X}^* = G_\text{c} \times G_\text{td}, \quad G_\text{c} = \Is{M}, \, G_\text{td} = \Is{Y} $$
where $G_\text{c}$ is a product of almost-connected non-compact center-free simple Lie groups and $G_\text{td}$ is totally disconnected. See Theorems 1.1, 1.6, 1.8 of  \cite{CM1}.

\begin{proof}[Proof of Theorem \ref{thm:local rigidity of uniform in general cat0-intro}]

In light of Lemma \ref{lem:local rigidity and finite index} it suffices therefore to prove that $\Gamma$ admits a locally rigid normal subgroup of finite index.

Let $\Gamma \le \Is{X}$ be a uniform lattice. Up to passing to a finite index normal subgroup and relying on Lemma \ref{lem:existence of splitting} we obtain a splitting of $\Gamma$ as a direct product of abstractly irreducible subgroups and an associated splitting of $X$. Treating each factor of $X$ separately and without loss of generality we may assume that $\Gamma$ is abstractly irreducible to begin with.

Recall the decomposition $\Is{X}^* = G_\text{c}\times G_\text{td} $ mentioned above. Up to passing to a further finite index normal subgroup of $\Gamma$ we may assume, to begin with, that $\Gamma \le G_\text{c}^\circ \times G_\text{td}$ where $G_\text{c}^\circ$ denotes the connected component of the identity of $G_\text{c}$. 
        
We treat separately three possible cases. If $G_\text{c}^\circ$ is trivial the conclusion follows from Theorem \ref{thm:local rigidity purely td cat0}. On the other hand, if $G_\text{td}$ is trivial we may apply the classical result of Weil. Namely, it is shown in  \cite{Weil2} that uniform lattices in connected semisimple Lie groups without compact factors are locally rigid, provided that the projection to every factor locally isomorphic to $\mathrm{SL}_2(\RR)$ is non-discrete. This additional condition is ensured by our assumptions. 
        
It remains to deal with the third case where both factors $G_\text{c}, G_\text{td}$ are non-trivial and $\Gamma$ is abstractly irreducible. We will rely on the superrigidity theorem of Monod \cite{Mo} for irreducible uniform lattices in products of locally compact groups. See theorem 8.4 of \cite{CM1} for a variant of that theorem adapted to our situation (we also refer to \cite{GKM} for more general results). 


Consider  any simple factor $S$  of $G_\text{c}^\circ$. As $\Gamma$ is abstractly irreducible the projection of $\Gamma$ to $S$ is non-discrete. It follows from the Borel density theorem (see e.g. chapter 5 of \cite{Rag}) that this projection is in fact dense. Similarly, $\Gamma$ projects non-discretely to every factor of $G_\text{td}$ corresponding to an irreducible factor of $Y$. Let $D \le G_\text{td}$ denote the product of the closures of these projections. It follows that  $\Gamma$ is an irreducible uniform  lattice in $G_\text{c}^\circ \times D$.  Moreover as $\Gamma \le \Is{X}$ is uniform the subgroup $D$ is acting cocompactly on $Y$.

Let $r \in \mathcal{R}(\Gamma,\Is{X})$ be a sufficiently small deformation of $\Gamma$ so that $r(\Gamma)$ is a uniform lattice in $\Is{X}$, as guaranteed by Theorem \ref{thm:local rigidity for lc groups}. The action of $r(\Gamma)$ on $X$ is minimal and without fixed points at infinity, as follows from the results of Caprace and Monod discussed in \S\ref{sub:cat0 lattices and their properties}.  In particular, for every simple factor $S$ of $G_\text{c}^\circ$  the action of $\Gamma$ on the corresponding irreducible symmetric space $M_S$  given by the projection of $r(\Gamma)$ into $S$ is clearly minimal and fixes no point of $\partial M_S$. 

Regarding $\Gamma$ as an irreducible lattice in $G_\text{c}^\circ \times D$  we are now precisely in a situation to apply Theorem 8.4 of \cite{CM1},  for every simple subgroup $S$ of $G_\text{c}^\circ$ and with respect to the given action of $\Gamma$ on $M_S$. Therefore this $\Gamma$-action extends to a continuous $G_\text{c}^\circ \times D$-action by isometries. In other words there is a continuous homomorphism $\alpha : S \to \Is{M_S}^\circ = S$ such that
$$ \pi_S \circ r = \alpha \circ \pi_S \in \mathcal{R}(\Gamma,S)$$
where $\pi_S : G_\text{c}^\circ \times D \to S$ is the projection. Since  $S$ is simple and $\alpha$ is a Lie group homomorphism with a dense image, it must be an automorphism of $S$ relying on \cite{omori1966homomorphic} or  \cite[Corollary 1.4]{bader2014equicontinuous}. 
According to  Lemma \ref{lem:small-is-inner} below, if $r$ is a sufficiently small deformation $\alpha$ must be inner.

We have showed that up to conjugation by an element of $G_\text{c}^\circ$ a sufficiently small deformation of $\Gamma$ may be assumed to be trivial on the $G_\text{c}^\circ$ factor. 
The result now follows from Theorem \ref{thm:local rigidity with totally disc factor} applied to the isometry group $\Is{X}^* = G_\text{c} \times G_\text{td}$.
\end{proof}

\begin{remark}
\label{remark: on normalizers in CAT0}
Assume that the space $X$ is reducible and that the lattice $\Gamma$ is abstractly irreducible. Let $G \le \Is{X}$ denote the product of the closures of the projections of $\Gamma$ to the different factors of $\Is{X}$. Therefore  $\Gamma$ can be regarded as an irreducible lattice in $G$. In this situation Theorem \ref{thm:local rigidity of uniform in general cat0-intro} holds in a slightly stronger sense. Indeed, relying on Proposition \ref{pro:locally rigid in and dense projections} we deduce that $\Gamma$ regarded as a lattice in $G$ is locally rigid in $N_{\Is{X}}(G)$.
\end{remark}

The following result concerning automorphisms of simple Lie groups was used in the  proof of Theorem \ref{thm:local rigidity of uniform in general cat0-intro} to show that a certain automorphism of a Lie group is inner.

\begin{lemma}\label{lem:small-is-inner}
Let $S$ be a connected center-free simple Lie group and $\Sigma\le S$ be a finite set generating a dense subgroup. There is an identity neighborhood
$U \subset S$ such that if \(T\) is an automorphism of $S$ satisfying $T(\sigma)\sigma^{-1}\in U$ for all $\sigma\in \Sigma$, then $T$ is inner. 
\end{lemma}

\begin{proof}
Recall that $\text{Aut}(S)^\circ=\text{Inn}(S)\cong \text{Ad}(S)\cong S$ and that $\text{Aut}(S)$ is canonically isomorphic to $\text{Aut}(\text{Lie}(S))$. Moreover, the analog statement  for the Lie algebra, obtained by replacing $\Sigma$\ with a spanning set of $\text{Lie}(S)$ and $U$ with a neighborhood of $0$ in the vector space $\text{Lie}(S)$, is obvious by the definition of the topology of $\text{Aut}(\text{Lie}(S))$, since $\text{Aut}(\text{Lie}(S))$ is a Lie group and $\text{Aut}(\text{Lie}(S))^\circ$ is open subgroup.
Now since $\langle\Sigma\rangle$ is dense in $S$, we can find $d=\text{dim}(S)$ words $W_1,\ldots,W_d$ in the generators $\Sigma$ which fall sufficiently close to the identity element so that $\log(W_i),~i=1,\ldots, d$ are well defined, and such that $\log(W_i),~i=1,\ldots, d$ span $\text{Lie}(S)$. This implies the lemma. 
\end{proof}

\subsection{Chabauty local rigidity for $\CAT$ groups}

Let $X$ be a proper geodesically complete $\CAT$ space with a cocompact group of isometries $\Is{X}$. The group $\Is{X}$ is compactly presented, as was mentioned in \S\ref{sub:cat0 lattices and their properties} above. Moreover any compact normal subgroup $N \nrm \Is{X}$ is trivial. Indeed, the subspace $X^N$ of $N$-fixed points  is $\Is{X}$-invariant, convex and non-empty. This discussion shows that Corollary \ref{cor:local rigidity in the Chabauty sense-intro} immediately follows from Theorem \ref{thm:Chabauty local rigidity for compacty presented groups}.

The following statement is a  direct consequence of local rigidity for $\CAT$ groups as in Theorem \ref{thm:local rigidity of uniform in general cat0-intro} and of Chabauty local rigidity as in Theorem \ref{thm:Chaubuty local rigidity}.

\begin{cor}
	\label{cor:occ for CAT0 groups}
	Let $X$ be a proper geodesically complete $\CAT$ space without Euclidean factors and with $\Is{X}$ acting cocompactly. Let  $\Gamma \le \Is{X}$ be a uniform lattice and assume that for every de Rham factor $Y$ of $X$ isometric
	to the hyperbolic plane the projection of $\Gamma$ to $\Is{Y}$
	is non-discrete. Then $\Gamma$ has the open conjugacy class property.
\end{cor}
\begin{proof}
The image closure map 
$\mathrm{C} : \mathcal{R}(\Gamma,\Is{X}) \to \Sub{\Is{X}}$ is a local homeomorphism at the point $r_0 \in \mathcal{R}(\Gamma,\Is{X})$ corresponding to the inclusion mapping  $r_0 : \Gamma  \hookrightarrow \Is{X}$. Theorem \ref{thm:local rigidity of uniform in general cat0-intro} provides us with an open neighborhood $\Omega$ of $r_0$ in $\mathcal{R}(\Gamma,\Is{X}) $ so that for every $r \in \Omega$ the subgroup $r(\Gamma)$  is  conjugate to $\Gamma$ in $G$. Therefore the Chabauty neighborhood $\mathrm{C}(\Omega)$ of $\Gamma$  consists of  conjugates.
\end{proof}

%

\section{First Wang's finiteness --- lattices containing a given lattice}
\label{sec:lattices containing a given lattice}

In this section we prove Theorem \ref{thm:W1KM-intro} which is a finiteness result for the number of lattices that contain a given lattice. This is a generalization of a theorem of H.C. Wang
\cite{wang_finitely}, and the proof that we give is inspired by Wang's original proof; see also Chapter 9 of \cite{Rag}.



\begin{definition}
	\label{def:jointly discrete}
	Let $G$ be a locally compact group. A family  of lattices in $G$ is \emph{jointly discrete}\footnote{\label{footnote:on jointly discrete} In the older literature such a family is called {\it uniformly discrete}. However, in the more recent literature, the notion of uniform discreteness is often used for the stronger property given in Definition \ref{def:uniformly discrete}.} if there is an identity neighborhood $U \subset G$ such that $\Gamma \cap U = \{e\}$ for every lattice $\Gamma$ in that family.
\end{definition}

Recall the notion of property $(KM)$ defined in the introduction. The following lemma relates this property to joint discreteness.

\begin{lemma}\label{lem:KM->UD}
	Let $G$ be group and $\mathcal{F}$ a family of lattices in $G$ with property $(KM)$. If $\mathcal{F}$ has a least element then it is jointly discrete.
\end{lemma}

\begin{proof}
	Let $U \subset G$ be an open neighborhood as in Definition \ref{def:KM property} and  $\Gamma\le G$ be a least element of $\mathcal{F}$. Since
	$\Gamma$ is of co-finite volume, the set 
	$$
	\{g\Gamma\in G/\Gamma:g\Gamma g^{-1}\cap U=\{e\}\}
	$$
	is relatively compact and nonempty.
	Thus, we can construct a Borel fundamental domain $\mathcal{D}$ for $\Gamma$\ in
	$G$ such that 
	$$
	\mathcal{D}_U:=\{ g\in \mathcal{D}:g\Gamma g^{-1}\cap U=\{e\}\}
	$$
	is relatively compact.
	
	Let $\Gamma'\le G$ be any  lattice  in the family $\mathcal{F}$, so that $\Gamma \le \Gamma'$. Then $\Gamma'$ admits
	a fundamental domain $\mathcal{D}'$ which is contained in $\mathcal{D}$.
	Because of property $(KM)$, there is an element $g\in \mathcal{D}'$ such that $g\Gamma'
	g^{-1}\cap U=\{e\}$. As $\Gamma\le\Gamma'$ we have that 
	$g\in\mathcal{D}_U$. Finally since $G$\ is locally compact and 
	$\mathcal{D}_U$ is relatively compact
	$$
	U_\Gamma:=\bigcap_{g \in \mathcal{D}_U}  g^{-1} U g
	$$
	is a neighborhood of the identity in $G$; see e.g. \cite[Theorem 2.4]{montgomery}. By construction the neighborhood $U_\Gamma$ 
	intersects trivially every lattice from the family $\mathcal{F}$.
\end{proof}

Theorem \ref{thm:W1KM-intro} is a consequence of Lemma \ref{lem:KM->UD} and the following result, which is of independent interest. 

\begin{thm}
	\label{thm:trivial centralizer implies finitely many superlattices}
	Let $G$ be a compactly generated locally compact group. Assume that every  lattice in $G$ has a trivial centralizer and let $\Gamma \le G$ be a finitely generated lattice. Then every jointly discrete family of lattices in $G$ all containing $\Gamma$ is finite.
\end{thm}


In the proof we rely on the notion of Chabauty topology as well as the Mahler--Chabauty compactness criterion, see \S\ref{sec:the Chabauty space of closed subgroups}. 
The discussion of these notions in \cite[Chapter 1]{Rag} required separability, but that assumption is in fact redundant.

\begin{proof}[Proof of Theorem \ref{thm:trivial centralizer implies finitely many superlattices}]
	Observe that the co-volume of the lattices $\Gamma'$ with $\Gamma \le  \Gamma' \le G$ and $\Gamma' \cap U = \{e\}$ is bounded from below, and since 
	$$ \text{co-vol}(\Gamma) = \left[\Gamma':\Gamma\right] 
	\cdot \text{co-vol}(\Gamma'),$$
	the index $\left[\Gamma':\Gamma\right]$ is bounded from above. We may therefore restrict our attention to lattices $\Gamma'$ as above with $\left[\Gamma':\Gamma\right] = d$ where $d \in \NN$ is fixed.
	Assume that there exists a sequence $\Gamma_n$ of pairwise distinct lattices such that 
	$$\Gamma \le \Gamma_n \le G, \quad  \left[\Gamma_n : \Gamma\right] = d, \quad \text{and} \quad \Gamma_n \cap U = \{e\}.$$
	We will arrive at a contradiction by  showing that $\Gamma_n$ has an infinite repetition. 
	
	Consider the actions of $\Gamma_n$ on the finite sets $\Gamma_n / \Gamma$ given by the natural homomorphisms $\varphi_n : \Gamma_n \to \mathrm{Sym}(\Gamma_n / \Gamma) \cong S_d$ for every $n \in \NN$. Denote $$\Gamma'_n = \ker \varphi_n \nrm \Gamma_n$$ so that  
	$$\Gamma'_n \le \Gamma \quad \text{and}  \quad \left[\Gamma : \Gamma'_n \right] = \frac{\left[\Gamma_n:\Gamma'_n\right]}{\left[\Gamma_n:\Gamma\right]}\le (d-1)!$$
	
	Note that since $\Gamma$ is finitely generated it has finitely many subgroups of every given finite index.  We now consider the group $\Delta \le \Gamma$ given by
	$$ \Delta = \bigcap_n \Gamma'_n$$
	and deduce from the above discussion that $\Delta$ is a finite index subgroup of $\Gamma$. 
	In particular $\Delta$ is finitely generated as well.
	In addition, for every pair of elements $\sigma \in \Delta$ and $\delta_n \in \Gamma_n$ we have that
	$$ \delta_n \sigma \delta_n^{-1} \in \delta_n \Delta \delta_n^{-1} \le \delta_n \Gamma'_n \delta_n^{-1} = \Gamma'_n \le \Gamma. $$
	
	We now apply the Mahler--Chabauty compactness criterion. Namely, from the assumption $\Gamma_n \cap U = \{e\}$ it follows that some subnet $\Gamma_{n_\alpha}$ converges in the Chabauty topology to a discrete subgroup $\Lambda \le G$ containing $\Gamma$ and such that
	$$
	\text{co-vol}(\Lambda) \le \lim_{\alpha} \text{co-vol}(\Gamma_{n_\alpha}) = d^{-1}\cdot \text{co-vol}(\Gamma).
	$$ 
	In particular, it follows that
	$\left[\Lambda : \Gamma \right] \ge d $. On the other hand $\Lambda$ is a lattice containing $\Gamma$ and hence $\Lambda$ is finitely generated.
	
	By the definition of Chabauty convergence, for every $\lambda \in \Lambda$ there is a $\delta_{n_\alpha} \in \Gamma_{n_\alpha}$ with $\delta_{n_\alpha} \to \lambda$. We obtain the following expression for every $\sigma \in \Delta$
	$$ \lim_{\alpha} \delta_{n_\alpha} \sigma \delta_{n_\alpha}^{-1} = \lambda \sigma \lambda^{-1}. $$
	Since the elements $\delta_{n_\alpha} \sigma \delta_{n_\alpha}^{-1}$ all belong to the discrete group $\Gamma$ this converging net must eventually stabilize. It other words
	$$ \delta_{n_\alpha} \sigma \delta_{n_\alpha}^{-1} = \lambda \sigma \lambda^{-1}$$ 
	holds for all $\alpha\ge\alpha_\sigma$. Applying this argument with $\sigma$ ranging over a finite generating set for $\Delta$ and in light of the fact that $\centralizer{G}{\Delta}$ is trivial we see that $\delta_{n_\alpha} = \lambda$ for all $\alpha$ sufficiently large.
	
	By the above argument, every generator of $\Lambda$ belongs to $\Gamma_{n_\alpha}$ and in particular $\Lambda \le \Gamma_{n_\alpha}$ for all  $\alpha$ sufficiently large. However the fact that $ \left[\Lambda : \Gamma \right] \ge \left[\Gamma_{n_\alpha} : \Gamma \right] $  shows that the net $\Gamma_{n_\alpha}$ must eventually stabilize, producing the required infinite repetition.
\end{proof}



\section{Second Wang's finiteness --- lattices of bounded covolume}
\label{sec:wangs finiteness theorem}

The classical finiteness theorem of Wang \cite{wangtopics} states that a connected semisimple Lie group without factors locally isomorphic to $\SL{2}{\RR}$ or $\SL{2}{\CC}$ admits only finitely many conjugacy classes of irreducible lattices of volume bounded by $v$ for any $v>0$. 

Recall Theorem \ref{thm:wang-finiteness-intro} of the introduction which is a suitable generalization of Wang's finiteness to the $\CAT$ context, restated here for the reader's convenience.  


\begin{theorem*}

Let $X$ be a proper geodesically complete $\CAT$ space without Euclidean
factors and with $\Is{X}$ acting cocompactly. Let $\mathcal{F}$  be a uniformly
discrete family of lattices in $\Is{X}$ so that every $\Gamma \in \mathcal{F}$
projects non-discretely to the isometry group of every hyperbolic plane factor.
Then $\mathcal{F}$ admits only finitely many conjugacy classes of lattices
with $\text{co-vol} \le v$ for every fixed $v > 0$.
\end{theorem*}

The proof follows rather immediately from a combination of Mahler--Chabauty compactness criterion and the open conjugacy class property for $\CAT$ groups established in \S\ref{sec:the Chabauty space of closed subgroups}.

\begin{proof}[Proof of Theorem \ref{thm:wang-finiteness-intro}]

Assume without loss of generality that the family $\mathcal{F}$ is closed under conjugation by elements of $\Is{X}$. In particular $\mathcal{F}$ is jointly discrete in the sense of Definition \ref{def:jointly discrete}, so that there is an identity neighborhood $ U \subset G$ with $\Gamma \cap U =\{e\}$ for all $\Gamma \in \mathcal{F}$.


Consider some fixed $v > 0$. The Mahler--Chabauty compactness criterion \cite[1.20]{Rag} says that the set of lattices in $\Sub{G}$ of co-volume at most $v$ and intersecting $U$ only at the identity element is compact. In particular $\mathcal{F}$ is relatively compact. 

On the other hand, every lattice  in the Chabauty closure  of $\mathcal{F}$ is uniform by \cite[1.12]{Rag} and therefore has the open conjugacy class property by Corollary \ref{cor:occ for CAT0 groups}. The required finiteness result follows.

\end{proof}



The following  consequence of Theorem \ref{thm:wang-finiteness-intro} covers many classical examples:

\begin{cor}
\label{cor:wangs finiteness for cocompact subgroup of cat0 isometries}
Let $X$ and $\Is{X}$ be as in Theorem \ref{thm:wang-finiteness-intro}. Let $G \le \Is{X}$ be a cocompact subgroup and $\mathcal{F}$ a uniformly discrete family of lattices in $G$. Then $\mathcal{F}$ admits only finitely many $\Is{X}$-conjugacy classes of lattices with $\text{co-vol} \le v$ for every fixed $v > 0$.
\end{cor}

Since $G$ is cocompact in $\Is{X}$, for a lattice $\Gamma \le G$ the covolumes of $\Gamma$ in $G$ and in $\Is{X}$ are proportional. Therefore to deduce Corollary \ref{cor:wangs finiteness for cocompact subgroup of cat0 isometries} from Theorem \ref{thm:wang-finiteness-intro} it suffices to show that $\mathcal{F}$ is uniformly discrete in $\Is{X}$. This follows from: 

\begin{lemma}
Let $H$ be a locally compact group, $G\le H$ a cocompact subgroup and $\mathcal{F}$ a family of subgroups which is uniformly discrete in $G$. Then $\mathcal{F}$ is uniformly discrete regarded as a family of lattices in $H$.
\end{lemma}   
   
\begin{proof}        
Suppose that $\mathcal{F}$ is uniformly discrete in $G$ with respect to the open subset    $U \subset G$. Say $ U = G \cap V$ with $V \subset H$ open.  Moreover let $K \subset H$ be a compact symmetric subset so that $H = GK$ and choose an open subset $W \subset V$ so that $k^{-1} W k \subset V$ for every $k \in K$, using e.g. \cite[Theorem 2.4]{montgomery}.  Then $\mathcal{F}$ is uniformly discrete in $H$ with respect to $W$.
Indeed, for $\Gamma \in \mathcal{F}$ and an element $h = gk$ with $h \in H, g \in G$ and $ k \in K$ we have the equation
$ \Gamma^h = (\Gamma^g)^k$.  
\end{proof}

\subsection{Linear groups over local fields}
\label{sub:groups over local fields}

Corollary \ref{cor:wangs finiteness for cocompact subgroup of cat0 isometries} immediately implies a Wang-type finiteness result for uniform lattices in simple Chevalley groups over non-Archimedean local fields. Such groups act on the associated Bruhat--Tits building $X$ which is a $\CAT$ space; see e.g. \cite[II.10.A]{BrHa}.

\begin{cor}
\label{cor:wang for simple groups over non-Archimedean field}
Let $G$ be a $k$-simple algebraic $k$-group with  $\mathrm{rank}_k(G) > 0$ where $k$ is  a non-Archimedean local field. Let $X$ be the Bruhat--Tits building associated to $G$. 
Fix an identity neighborhood $U \subset G$ and some $v > 0$. Then there are only finitely many $\Is{X}$-conjugacy classes of lattices $\Gamma$ in $G$ with $\text{co-vol}(\Gamma) < v$ and $\Gamma^g \cap U = \{e\}$ for every $g \in G$.
\end{cor}


%


\begin{remark}
\label{remark:about Tits theorem}
If $\mathrm{rank}_k(G) \ge 2$ then  $\Is{X} = \mathrm{Aut}(G)$ and the statement of Corollary \ref{cor:wang for simple groups over non-Archimedean field} can be somewhat simplified. See \cite[Proposition C.1]{bader2016linearity}.
\end{remark}

\section{Invariant random subgroups}
\label{sec:invariant random subgroups}

Let $G$ be a locally compact group. It is acting on its space of closed subgroups $\Sub{G}$  by conjugation and this action is continuous. 

\begin{definition}
	\label{def:invariant random subgroup}
	An \textbf{invariant random subgroup} of $G$ is a $G$-invariant Borel probability measure on $\Sub{G}$. Let $\IRS{G}$ denote the space of all invariant random subgroups of $G$ equipped with the weak-$*$ topology.
\end{definition}

The space $\IRS{G}$ is compact. For a further discussion of invariant random subgroups the reader is referred to \cite{7S, gelander2015lecture}.

Let $\mathrm{ULat}(G)$ denote the space of all the uniform lattices in $G$ with the induced topology  from $\Sub{G}$. There is a natural map
$$ 
\mathrm{I} : \mathrm{ULat}(G) \to \IRS{G}, \quad \mathrm{I}(\Gamma) = \mu_{\Gamma},
$$ 
where the  invariant random subgroup $\mu_\Gamma \in \IRS{G}$ is obtained by pushing forward the $G$-invariant probability measure from $G/\Gamma$ to $\Sub{G}$ via the map $g\Gamma \mapsto g\Gamma g^{-1}$.



\begin{prop}
	\label{prop:Ulat space and irs}
If $G$ is compactly generated then the map $\mathrm{I} : \mathrm{ULat}(G) \to \mathrm{IRS}(G)$ is continuous.
\end{prop}

The proof of this  proposition depends on showing that the covolume function on the space $\mathrm{ULat}(G)$ is continuous, see Proposition \ref{prop:covolume is continous} below.

Recall that a uniform lattice $\Gamma$ in $G$ is  topologically locally rigid. Therefore  Proposition \ref{prop:C is continuous at a uniform lattice}  implies that there exists a neighborhood  $\Omega$ of the inclusion homomorphism $r_0$ in $ \mathcal{R}(\Gamma,G)$ on which is image closure map $C$ is continuous and takes values in $\mathrm{ULat}(G)$. 

\begin{cor}
	\label{cor:representation space and irs}
Let $\Gamma$ be a uniform lattice in  $G$.  If $G$ is compactly generated then there is a neighborhood $\Omega \subset \mathcal{R}(\Gamma,G)$ of the inclusion homomorphism $r_0 $ so that the composition 
	$$ \Omega \xrightarrow{\mathrm{C}} \mathrm{ULat}(G) \xrightarrow{	\mathrm{I}}  \mathrm{IRS}(G)$$ is continuous. 
\end{cor}

Note that  local rigidity of $\Gamma$ is equivalent to $\mathrm{I} \circ \mathrm{C}$ being locally constant at $r_0$.

\subsection{Co-volumes of lattices}
\label{sub:the covolume is continuous}


Let $G$ be any locally compact  group with a choice $\mu$ of a Haar measure. The co-volume function is well-defined on the space $\mathrm{ULat}(G)$ of uniform lattices in $G$.
%

\begin{prop}
	\label{prop:covolume is continous}
If $G$ is compactly generated then the co-volume function is continuous on $\mathrm{ULat}(G)$.
\end{prop}

The proof of Proposition \ref{prop:covolume is continous} will rely on the following lemma, which is related to Serre's geometric volume formula \cite[Chapter 1.5]{bass2001tree}. 


\begin{lemma}
	\label{lem:volume function of locally constant}
	Let $K \subset G$ be a fixed compact set. Given a uniform lattice $\Gamma \in \mathrm{ULat}(G)$ let 
	$ \pi_\Gamma : G \to G/\Gamma $ denote the projection and  
	let $\mu_\Gamma$  denote  the induced  measure on $G/\Gamma$. Then the following function
	$$v_K : \mathrm{ULat}(G) \to \RR_{\ge 0}, \quad v_K(\Gamma) = \mu_\Gamma(\pi_\Gamma(K)). $$
is continuous. 
\end{lemma}

\begin{proof}
	
	Given a compact subset $Q$ of $G$, consider the function
	$$
	n_Q : G \times \mathrm{ULat}(G) \to \NN_{\ge 0 }, \quad n_Q(x,\Gamma) = \left|\{\gamma \in  \Gamma \: : \: x \gamma \in Q \}\right|. 
	$$
	Note that $n_Q(\cdot, \Gamma)$ is $\Gamma$-invariant from the right, and so descends to a well defined function on $G/\Gamma$.
	
	In terms of the function $n_Q$ we obtain the formula:
	\begin{align*}
	v_Q(\Gamma) &= \mu_\Gamma(\pi_\Gamma(Q)) = 
	\int_{G/ \Gamma} \mathbbm{1}_{\pi_\Gamma(Q)}(x) \mathrm{d} \mu_\Gamma (x) = \\
	&= \int_{G/\Gamma} \frac{1}{n_Q(x,\Gamma)}\left(\sum_{\gamma \in \Gamma } \mathbbm{1}_{Q}(\tilde{x}\gamma)\right)   \mathrm{d} \mu_\Gamma (x) = 
	\int_G \frac{1}{n_Q(x,\Gamma)} \mathbbm{1}_{Q}(x) \mathrm{d} \mu (x) = \\
	&= \int_Q \frac{1}{n_Q(x,\Gamma)} \mathrm{d} \mu (x),
	\end{align*}
	where $\tilde{x}$ is an arbitrary lift of $x$.
	
Let $\Gamma_0 \in \mathrm{ULat}(G)$ be some fixed uniform lattice.	The following two facts\footnote{The convergence $\Gamma \to \Gamma_0$ is understood in the sense of nets. If $G$ is metrizable then it is enough to consider sequences.}  are consequences of the definition of the Chabauty topology:
	\begin{enumerate}
		\item  $\lim_{\Gamma\to \Gamma_0}n_K(x,\Gamma)\le n_K(x,\Gamma_0)$, for all $x \in G$.
		\item If $Q$ is a compact set containing $K$ in its interior, then $$ \lim_{\Gamma\to \Gamma_0}n_Q(x,\Gamma)\ge n_K(x,\Gamma_0)$$
		for all $ x \in G$.
	\end{enumerate}  
	
	Since $n_K(x,\Gamma) \ge 1$ for all $x \in K$, it follows from Item $(1)$, the above volume formula and the dominated convergence theorem that   $v_K(\Gamma_0)\le\lim_{\Gamma\to \Gamma_0} v_K(\Gamma)$. 
	
	To obtain the reverse inequality,  let $\varepsilon>0$ and pick a compact set $Q$ containing $K$ in its interior and satisfying $\mu(Q\setminus K ) < \varepsilon$. Note that this implies
	$$
	\int_{Q\setminus K}\frac{1}{n_Q(x,\Gamma)} \mathrm{d} \mu (x)<\epsilon
	$$  
	for all $\Gamma \in \mathrm{ULat}(G)$.
	It follows from Item $(2)$ and the volume formula that 
	$$
	v_K(\Gamma_0) =\int_K \frac{1}{n_K(x,\Gamma_0)} \mathrm{d} \mu (x)\ge\lim_{\Gamma\to \Gamma_0}\int_K \frac{1}{n_Q(x,\Gamma)} \mathrm{d} \mu (x).
	$$                     
	Combining these two last inequalities, we get 
	$$
	v_K(\Gamma_0)+\epsilon>\lim_{\Gamma\to \Gamma_0}\int_Q \frac{1}{n_Q(x,\Gamma)} \mathrm{d} \mu (x)=\lim_{\Gamma \to \Gamma_0}v_Q(\Gamma)\ge\lim_{\Gamma\to \Gamma_0}v_K(\Gamma). $$           
	The continuity of the map $v_K$ at the point $\Gamma_0$ follows by letting $\varepsilon \to 0$.	
\end{proof}

\begin{proof}[Proof of Proposition \ref{prop:covolume is continous}]
Let $\Gamma$ be a uniform lattice in $G$. Since $G$ is compactly generated it follows from Proposition \ref{prop:condition for a uniform lattice to admit a Chabauty neighborhood of cocompact subgroups} and  Remark \ref{rem:uniformly cobounded} that there is a Chabauty neighborhood $\Gamma \in \Omega \subset \mathrm{ULat}(G)$ and a compact subset $K \subset G$ so that $\Gamma_0 K = G$ for every $\Gamma_0 \in \Omega$.
Taking into account Lemma \ref{lem:volume function of locally constant} and as soon as $\Gamma' \in \Omega$ we obtain
	$$ \text{co-vol}(G/\Gamma) = v_K(\Gamma) = \lim_{\Gamma' \to \Gamma} v_K(\Gamma') = \lim _{\Gamma' \to \Gamma} \text{co-vol}(G/\Gamma')  $$
as required.
\end{proof}

\begin{remark}
If $(\Gamma_\alpha)_{\alpha}$ is a net of lattices in $G$ converging in the Chabauty topology to the lattice $\Gamma$  then
	$$ \text{co-vol}(G/\Gamma) \le  \liminf _\alpha  \text{co-vol}(G/\Gamma_\alpha).  $$
This does not require  compact generation or co-compactness  \cite[I.1.20]{Rag}. Proposition \ref{prop:covolume is continous} can be regarded as a refinement of this fact.
	\end{remark}

\begin{proof}[Proof of Proposition \ref{prop:Ulat space and irs}]
Let $G$ be a compactly generated group and  $\Gamma_\alpha \in \mathrm{ULat}(G)$  be a net\footnote{Once more, if $G$ is metrizable we may speak of sequences instead of nets.} of uniform lattices in $G$ indexed by $\alpha \in A$ and converging to $\Gamma \in \mathrm{ULat}(G)$.
	
	Let $\mathcal{D}$ be a Borel fundamental domain for $\Gamma$ in $G$.  There  is an increasing sequence of compact subsets $K_m \subset \mathcal{D}$ so that the Haar measure of $\mathcal{D} \setminus K_m$ goes to zero as $ m\to \infty$. Moreover the argument on p. 28 of \cite{Rag} shows  that every $K_m$ injects into  $G/\Gamma_{\alpha}$ for all $\alpha > \beta_1 = \beta_1(m) $ sufficiently large in $A$.  
	
	%
	Let $\varepsilon > 0$ be arbitrary and take $l = l(\varepsilon)$ so that $\mu(\mathcal{D}\setminus K_{l}) < \varepsilon$. Let $\mathcal{D}_\alpha$ be a Borel fundamental domain for $\Gamma_\alpha$ in $G$ for every $\alpha \in A$. The subsets $\mathcal{D}_\alpha$ can be chosen in such a way that $K_{l} \subset \mathcal{D}_\alpha$ for all $\alpha > \beta_1(l)$. 
	
	Let $\mu$ be the Haar measure on $G$ and $\mu_{|\mathcal{D}}$ its restriction to $\mathcal{D}$. Then $\mu_{|\mathcal{D}}$ projects to the $G$-invariant probability measure on $G/\Gamma$. Similarly let $\mu_{|\mathcal{D}_\alpha}$ denote the restrictions of $\mu$ to $\mathcal{D}_\alpha$ for $\alpha \in A$.  Write
	$$ \mu_{|\mathcal{D}} = \mu_{|K_{l}} + \mu_{|\mathcal{D} \setminus K_{l}}, \quad \mu_{|\mathcal{D}_\alpha} = \mu_{|K_{l}} + \mu_{|\mathcal{D}_\alpha \setminus K_{l}} $$
	Recall that $\mu(\mathcal{D}\setminus K_{l}) < \varepsilon$ by the choice of $l$.  In addition,  the covolume function is continuous at the point $\Gamma \in \mathrm{ULat}(G)$ according to Proposition \ref{prop:covolume is continous} and so 
	$$\mu(\mathcal{D}_\alpha) = \text{co-vol}(G/\Gamma_\alpha) \xrightarrow{\alpha \in A} \text{co-vol}(G/\Gamma) = \mu(\mathcal{D}). $$ 
	This implies that $\mu(\mathcal{D}_\alpha \setminus K_{l}) < 2\varepsilon$ for all $\alpha > \beta_2 = \beta_2(\varepsilon)$ in $A$.
	
	Consider the maps $\varphi_{\alpha}$ defined for every $\alpha \in A$ 
	$$ \varphi_{\alpha} : K_{l}  \to \Sub{G}, \quad \varphi_{\alpha}(k) = k \Gamma_\alpha k^{-1} .$$
	Since $K_{l}$ is compact, the functions $\varphi_\alpha$ converge uniformly to the function $\varphi : K_{l} \to \Sub{G}$ where $\varphi(k) = k \Gamma k^{-1}$. In particular
	$$ (\varphi_\alpha)_* \mu_{|K_{l}} \xrightarrow{\alpha \in A} (\varphi)_* \mu_{|K_{l}}. $$
	Taking $\varepsilon \to 0$ in the above argument gives the continuity of $\mathrm{I}$ at the point $\Gamma$.
\end{proof}

\bibliographystyle{alpha}
\bibliography{localrig}

\newcommand{\etalchar}[1]{$^{#1}$}
\begin{thebibliography}{ABB{\etalchar{+}}12}

\bibitem[ABB{\etalchar{+}}12]{7S}
Miklos Abert, Nicolas Bergeron, Ian Biringer, Tsachik Gelander, Nikolay
  Nikolov, Jean Raimbault, and Iddo Samet.
\newblock On the growth of $ l^{2}$-invariants for sequences of lattices in lie
  groups.
\newblock {\em arXiv preprint arXiv:1210.2961}, 2012.

\bibitem[BCL16]{bader2016linearity}
Uri Bader, Pierre-Emmanuel Caprace, and Jean L{\'e}cureux.
\newblock On the linearity of lattices in affine buildings and ergodicity of
  the singular cartan flow.
\newblock {\em arXiv preprint arXiv:1608.06265}, 2016.

\bibitem[BF13]{bader2013algebraic}
Uri Bader and Alex Furman.
\newblock Algebraic representations of ergodic actions and super-rigidity.
\newblock {\em arXiv preprint arXiv:1311.3696}, 2013.

\bibitem[BF14]{bader2014boundaries}
Uri Bader and Alex Furman.
\newblock Boundaries, rigidity of representations, and {L}yapunov exponents.
\newblock {\em arXiv preprint arXiv:1404.5107}, 2014.

\bibitem[BFS13]{bader2013integrable}
Uri Bader, Alex Furman, and Roman Sauer.
\newblock Integrable measure equivalence and rigidity of hyperbolic lattices.
\newblock {\em Inventiones mathematicae}, 194(2):313--379, 2013.

\bibitem[BG04]{Bergeron-Gelander}
Nicolas Bergeron and Tsachik Gelander.
\newblock A note on local rigidity.
\newblock {\em Geometriae Dedicata}, 107(1):111--131, 2004.

\bibitem[BG14]{bader2014equicontinuous}
Uri Bader and Tsachik Gelander.
\newblock Equicontinuous actions of semisimple groups.
\newblock {\em arXiv preprint arXiv:1408.4217}, 2014.

\bibitem[BH99]{BrHa}
Martin~R. Bridson and Andr{\'e} Haefliger.
\newblock {\em Metric spaces of non-positive curvature}, volume 319.
\newblock Springer Science \& Business Media, 1999.

\bibitem[BK90]{bass1990uniform}
Hyman Bass and Ravi Kulkarni.
\newblock Uniform tree lattices.
\newblock {\em Journal of the American Mathematical Society}, 3(4):843--902,
  1990.

\bibitem[BL01]{bass2001tree}
Hyman Bass and Alexander Lubotzky.
\newblock {\em Tree lattices}.
\newblock Springer, 2001.

\bibitem[BM00a]{burger2000groups}
Marc Burger and Shahar Mozes.
\newblock Groups acting on trees: from local to global structure.
\newblock {\em Publications Math{\'e}matiques de l'IH{\'E}S}, 92:113--150,
  2000.

\bibitem[BM00b]{burger2000lattices}
Marc Burger and Shahar Mozes.
\newblock Lattices in product of trees.
\newblock {\em Publications Math{\'e}matiques de l'IH{\'E}S}, 92:151--194,
  2000.

\bibitem[Bou04]{bourbaki2004integration}
Nicolas Bourbaki.
\newblock {\em Integration II}.
\newblock Springer, 2004.

\bibitem[Cal61]{Calabi}
Eugenio Calabi.
\newblock On compact {R}iemannian manifolds with constant curvature.
\newblock In {\em Proc. Sympos. Pure Math I}, volume~3, pages 155--180, 1961.

\bibitem[CdlH14]{cor_har}
Yves Cornulier and Pierre de~la Harpe.
\newblock Metric geometry of locally compact groups.
\newblock {\em arXiv preprint arXiv:1403.3796}, 2014.

\bibitem[CM09a]{CM2}
Pierre-Emmanuel Caprace and Nicolas Monod.
\newblock Isometry groups of non-positively curved spaces: discrete subgroups.
\newblock {\em Journal of Topology}, 2(4):701--746, 2009.

\bibitem[CM09b]{CM1}
Pierre-Emmanuel Caprace and Nicolas Monod.
\newblock Isometry groups of non-positively curved spaces: structure theory.
\newblock {\em Journal of Topology}, 2(4):661--700, 2009.

\bibitem[CW17]{caprace2017indicability}
Pierre-Emmanuel Caprace and Phillip Wesolek.
\newblock Indicability, residual finiteness, and simple subquotients of groups
  acting on trees.
\newblock {\em arXiv preprint arXiv:1708.04590}, 2017.

\bibitem[FL08]{foertsch2008rham}
Thomas Foertsch and Alexander Lytchak.
\newblock The de rham decomposition theorem for metric spaces.
\newblock {\em Geometric and Functional Analysis}, 18(1):120--143, 2008.

\bibitem[Gel04]{HV}
Tsachik Gelander.
\newblock Homotopy type and volume of locally symmetric manifolds.
\newblock {\em Duke Mathematical Journal}, 124(3):459--515, 2004.

\bibitem[Gel15a]{WUD}
Tsachik Gelander.
\newblock {Kazhdan-Margulis} theorem for invarant random subgroups.
\newblock {\em arXiv preprint arXiv:1510.05423}, 2015.

\bibitem[Gel15b]{gelander2015lecture}
Tsachik Gelander.
\newblock A lecture on invariant random subgroups.
\newblock {\em arXiv preprint arXiv:1503.08402}, 2015.

\bibitem[GKM08]{GKM}
Tsachik Gelander, Anders Karlsson, and Gregory~A. Margulis.
\newblock Superrigidity, generalized harmonic maps and uniformly convex spaces.
\newblock {\em Geometric and Functional Analysis}, 17(5):1524--1550, 2008.

\bibitem[GKM16]{glasner}
Yair Glasner, Daniel Kitroser, and Julien Melleray.
\newblock From isolated subgroups to generic permutation representations.
\newblock {\em arXiv preprint arXiv:1601.07538}, 2016.

\bibitem[GR70]{Gar-Rag}
Howard Garland and MS~Raghunathan.
\newblock Fundamental domains for lattices in $\mathbb{R} \mathrm{-rank} $ 1
  semisimple {Lie} groups.
\newblock {\em Annals of Mathematics}, pages 279--326, 1970.

\bibitem[Kle99]{Kl}
Bruce Kleiner.
\newblock The local structure of length spaces with curvature bounded above.
\newblock {\em Mathematische Zeitschrift}, 231(3):409--456, 1999.

\bibitem[KM68]{KM}
D.~Kazhdan and G.~Margulis.
\newblock A proof of {S}elberg's hypothesis.
\newblock {\em Math. Sbornik}, 75(117):162--168, 1968.

\bibitem[KM08]{moller}
Bernhard Kr{\"o}n and R{\"o}gnvaldur~G M{\"o}ller.
\newblock Analogues of cayley graphs for topological groups.
\newblock {\em Mathematische Zeitschrift}, 258(3):637--675, 2008.

\bibitem[Kra71]{kramarev1971topological}
B.~A. Kramarev.
\newblock On topological {FC}-groups.
\newblock {\em Siberian Mathematical Journal}, 12(2):266--271, 1971.

\bibitem[Lub91]{lubotzky1991lattices}
Alexander Lubotzky.
\newblock Lattices in rank one lie groups over local fields.
\newblock {\em Geometric and Functional Analysis}, 1(4):405--431, 1991.

\bibitem[Mac64]{macbeath}
Murray Macbeath.
\newblock Groups of homeomorphisms of a simply connected space.
\newblock {\em Annals of Mathematics}, pages 473--488, 1964.

\bibitem[Mar91]{Margulis}
Gregori~A. Margulis.
\newblock {\em Discrete subgroups of semisimple Lie groups}, volume~17.
\newblock Springer Science \& Business Media, 1991.

\bibitem[Mon06]{Mo}
Nicolas Monod.
\newblock Superrigidity for irreducible lattices and geometric splitting.
\newblock {\em Journal of the American Mathematical Society}, 19(4):781--814,
  2006.

\bibitem[MZ55]{montgomery}
Deane Montgomery and Leo Zippin.
\newblock {\em Topological transformation groups}, volume~1.
\newblock Interscience Publishers New York, 1955.

\bibitem[Omo66]{omori1966homomorphic}
Hideki Omori.
\newblock Homomorphic images of lie groups.
\newblock {\em Journal of the Mathematical Society of Japan}, 18(1):97--117,
  1966.

\bibitem[Pap05]{Papa}
Athanase Papadopoulos.
\newblock {\em Metric spaces, convexity and nonpositive curvature}, volume~6.
\newblock European mathematical society, 2005.

\bibitem[Pon66]{pontry}
Lev~Semenovich Pontryagin.
\newblock {\em Topological groups}.
\newblock Gordon and Breach, 1966.

\bibitem[Rag72]{Rag}
Madabusi~Santanam Raghunathan.
\newblock {\em Discrete subgroups of Lie groups}, volume~68.
\newblock Springer Verlag, 1972.

\bibitem[Sel60]{Selberg}
Atle Selberg.
\newblock On discontinuous groups in higher-dimensional symmetric spaces.
\newblock In {\em Contributions to function theory}, pages 147--164. Bombay,
  1960.

\bibitem[Sha00]{shalom2000rigidity}
Yehuda Shalom.
\newblock Rigidity of commensurators and irreducible lattices.
\newblock {\em Inventiones mathematicae}, 141(1):1--54, 2000.

\bibitem[SW13]{shalom2013commensurated}
Yehuda Shalom and George~A Willis.
\newblock Commensurated subgroups of arithmetic groups, totally disconnected
  groups and adelic rigidity.
\newblock {\em Geometric and Functional Analysis}, 23(5):1631--1683, 2013.

\bibitem[Ush63]{ushakov1963topological}
V.I. Ushakov.
\newblock Topological {$\overline{FC}$}-groups.
\newblock {\em Sib. Mat. Zh}, 4:1162--1174, 1963.

\bibitem[Wan67]{wang_finitely}
H.C. Wang.
\newblock On a maximality property of discrete subgroups with fundamental
  domain of finite measure.
\newblock {\em American Journal of Mathematics}, 89(1):124--132, 1967.

\bibitem[Wan72]{wangtopics}
H.C. Wang.
\newblock Topics on totally discontinuous groups, symmetric spaces, edited by
  {W. Boothby and G. Weiss}, 1972.

\bibitem[War54]{ward1954partially}
LE~Ward.
\newblock Partially ordered topological spaces.
\newblock {\em Proceedings of the American Mathematical Society},
  5(1):144--161, 1954.

\bibitem[Wei60]{We60}
Andr{\'e} Weil.
\newblock On discrete subgroups of {L}ie groups.
\newblock {\em Annals of Mathematics}, pages 369--384, 1960.

\bibitem[Wei62]{Weil2}
Andr{\'e} Weil.
\newblock On discrete subgroups of {L}ie groups {(II)}.
\newblock {\em Annals of Mathematics}, pages 578--602, 1962.

\bibitem[Wes15]{wesolek}
Phillip Wesolek.
\newblock Elementary totally disconnected locally compact groups.
\newblock {\em Proceedings of the London Mathematical Society},
  110(6):1387--1434, 2015.

\end{thebibliography}

\end{document}